\documentclass[11 pt]{amsart}

\usepackage{amssymb}
\usepackage{amsmath}
\usepackage{amsthm}
\usepackage{amsfonts}
\usepackage{amsxtra}
\usepackage[all,2cell]{xy}
\usepackage[unicode=true, pdfusetitle,
 bookmarks=true,bookmarksnumbered=false,
 breaklinks=false,
 backref=false,
 colorlinks=true,
 linkcolor=blue,
 citecolor=blue,
 urlcolor=blue,
 final
]{hyperref}

\usepackage{bookmark}

\usepackage[capitalise]{cleveref}

\newcommand{\newrefformat}[2]{}

\theoremstyle{plain}   
\newtheorem{thm}{Theorem}[section] 
\newtheorem{cor}{Corollary}[section]
\makeatletter\let\c@cor\c@thm\makeatother
\newtheorem{lemma}{Lemma}[section]
\makeatletter\let\c@lemma\c@thm\makeatother
\newtheorem{prop}{Proposition}[section]
\makeatletter\let\c@prop\c@thm\makeatother

\makeatletter\let\c@claim\c@thm\makeatother
\newtheorem{thrm}{Theorem}[section]
\makeatletter\let\c@thrm\c@thm\makeatother

\theoremstyle{definition}

\newtheorem{defn}{Definition}[section]
\makeatletter\let\c@defn\c@thm\makeatother
\newtheorem{const}{Construction}[section]
\makeatletter\let\c@const\c@thm\makeatother
\newtheorem{notn}{Notation}[section]
\makeatletter\let\c@notn\c@thm\makeatother
\newtheorem{outline}{Proof Outline}[section]
\makeatletter\let\c@outline\c@thm\makeatother
\newtheorem{recoll}{Recollection}[section]
\makeatletter\let\c@recoll\c@thm\makeatother

\theoremstyle{remark}

\newtheorem{rem}{Remark}[section]
\makeatletter\let\c@rem\c@thm\makeatother
\newtheorem{ex}{Example}[section]
\makeatletter\let\c@ex\c@thm\makeatother
\newtheorem{observation}{Observation}[section]
\makeatletter\let\c@observation\c@thm\makeatother

\makeatletter
\let\c@equation\c@thm
\numberwithin{equation}{section}
\makeatother

\crefname{lemma}{Lemma}{Lemmas}
\crefname{thm}{Theorem}{Theorems}
\crefname{defn}{Definition}{Definitions}
\crefname{notn}{Notation}{Notations}
\crefname{const}{Construction}{Constructions}
\crefname{prop}{Proposition}{Propositions}
\crefname{rem}{Remark}{Remarks}
\crefname{cor}{Corollary}{Corollaries}
\crefname{equation}{Diagram}{Diagrams}
\crefname{ex}{Example}{Examples}
\crefname{observation}{Observation}{Observations}

\newcommand{\cC}{\mathcal{C}}
\newcommand{\cD}{\mathcal{D}}
\newcommand{\V}{\mathcal{V}}

\newcommand{\A}{\mathcal{A}}
\newcommand{\B}{\mathcal{B}}
\newcommand{\F}{\mathcal{F}}
\newcommand{\Spec}{\mathbf{Spec}}
\newcommand{\E}{\mathcal{E}}
\newcommand{\D}{\mathcal{D}}
\newcommand{\cG}{\mathcal{G}}
\newcommand{\sma}{\wedge}
\newcommand{\fto}{\xrightarrow}
\newcommand{\xto}{\xrightarrow}
\newcommand{\inv}{{-1}}
\newcommand{\op}{\text{op}}
\newcommand{\cat}{\mathcal{C}\hspace{-.1em}\mathit{at}}
\newcommand{\id}{\textrm{id}}

\DeclareMathOperator{\map}{map}
\newcommand{\wed}{\vee}
\newcommand{\abs}[1]{\left|{#1}\right|}
\newcommand{\diag}{\textit{diag}}
\newcommand{\ul}{\underline}
\newcommand{\ol}{\overline}
\newcommand{\wreath}{\mathord{\int}}
\newcommand{\inc}{\iota}
\newcommand{\cO}{\mathcal{O}}
\newcommand{\bM}{\mathbf{M}}
\newcommand{\bN}{\mathbf{N}}

\newcommand{\sdot}{{S^{()}_{\bullet,\dots,\bullet}}}

\newcommand{\mb}{\mathbf}

\newcommand{\sset}{\mathit{s}\set}
\newcommand{\ssset}{\mathit{ss}\set}
\newcommand{\scat}{\mathit{s}\cat}

\newcommand{\set}{\mathcal{S}\hspace{-.1em}\mathit{et}}

\newcommand{\mcsmc}{\mathbf{SMC}}
\newcommand{\mcwald}{\mathbf{Wald}}
\newcommand{\mcdewald}{\mathbf{Wald}_{\vee}}
\newcommand{\EMk}{\mathbb{K}_{\text{SMC}}}
\newcommand{\waldk}{{\mathbb{K}_{\text{Wald}}}}
\DeclareMathOperator{\Ar}{Ar}
\DeclareMathOperator{\Ob}{Ob}
\DeclareMathOperator{\Sq}{Sq}

\newcommand{\Ecat}{\E_*\text{-}\cat}
\newcommand{\Inj}{\textrm{Inj}}
\newcommand{\Gcat}{\cG_*\text{-}\cat}
\newcommand{\wecat}{\cat_{we}}
\newcommand{\Ewecat}{\E_*\text{-}\wecat}
\newcommand{\mcsmcwe}{\mcsmc_{we}}
\newcommand{\fps}[1]{\underline{{#1}}_*}
\newcommand{\hrto}{\hookrightarrow}
\newcommand{\hlto}{\hookleftarrow}

\UseAllTwocells

\title{A multiplicative comparison of Segal and Waldhausen K-Theory}
\author{Anna Marie Bohmann}
\address{
Department of Mathematics, 
Vanderbilt University,
Nashville, TN 37235 USA
}
\email{am.bohmann@vanderbilt.edu }
\author{Ang\'elica M. Osorno}
\address{
Department of Mathematics,
Reed College,
Portland, OR 97202 USA
}
\email{aosorno@reed.edu}

\begin{document}

\begin{abstract}
In this paper, we establish a multiplicative equivalence between two multiplicative algebraic $K$-theory constructions, Elmendorf and Mandell's version of Segal's $K$-theory and Blumberg and Mandell's version of Waldhausen's $S_\bullet$ construction.  This equivalence implies that the ring spectra, algebra spectra, and module spectra constructed via these two classical algebraic $K$-theory functors are equivalent as ring, algebra or module spectra, respectively.  It also allows for comparisions of spectrally enriched categories constructed via these definitions of $K$-theory.  As both the Elmendorf--Mandell and Blumberg--Mandell multiplicative versions of $K$-theory encode their multiplicativity in the language of multicategories, our main theorem is that there is multinatural transformation relating these two symmetric multifunctors that lifts the classical functor from Segal's to Waldhausen's construction.  Along the way, we provide a slight generalization of the Elmendorf--Mandell construction to symmetric monoidal categories.
\end{abstract}

\maketitle

\section{Introduction}

Algebraic $K$-theory is a powerful invariant that connects number theory, algebraic geometry, geometric topology and homotopy theory.  It has many incarnations, reflecting these varied mathematical uses, but all share the basic underlying idea of ``splitting'' some kind of ``sum'' operation, whether that be direct sum of vector bundles or of modules, disjoint union of spaces, or a more exotic structure.  The first definitions of algebraic $K$-theory were the constructions of the zeroth and first algebraic $K$-groups by Grothendieck and Bass--Schanuel in the late 50s and early 60s.

In the late 60s, Quillen gave the first formulation of higher $K$-groups.  The fundamental philosophy behind his formulation is that higher algebraic $K$-theory should arise as the higher homotopy groups of a space, or more precisely, of an \emph{infinite loop space}: a space with extra structure making it equivalent to a \emph{spectrum}.   Thus algebraic $K$-theory creates interesting invariants by creating spectra.   Hence, from a homotopy theory perspective, algebraic $K$-theory provides a tool for constructing spectra.  From a more geometric or number theoretic point of view, spectra are a tool for constructing rich groups of invariants. These ideas put $K$-theory at the center of a fruitful mathematical symbiosis.

After Quillen's initial work, the  70s saw a flourishing of ``machines'' for building $K$-theory spectra out of algebraic or categorical data.  Segal made good on this idea in ``Categories and cohomology theories,'' \cite{segal} which builds a $K$-theory spectrum from any category with a symmetric monoidal product.  Contemporaneously, May's operadic technology provided another method for building $K$-theory spectra from suitable categorical input.  Slightly later, in the early 80s, Waldhausen produced another method for constructing $K$-theory spectra from a very flexible form of categorical input data, now known as \emph{Waldhausen categories.}  Additionally, Waldhausen provided a direct comparison of his construction with that of Segal.

The May--Thomason theorem \cite{maythomason} tells us that all of the May and Segal constructions of spectra from categorical data are equivalent:  May and Thomason provide a way to compare the input data for these constructions and prove that any possible way of building spectra from reasonable categorical data will produce equivalent output.  This is a fundamental result.  It gives homotopy theorists the flexibility to work with a wide variety of constructions depending on the situation at hand.

It has been clear from the origins of the subject that understanding algebraic structures on spectra is crucial to performing  research in homotopy theory.  The most basic such structure is some sort of multiplication ``up to homotopy,'' or better yet, ``up to coherent homotopy.''  There are now a number of ways to make sense of this idea.  For example, modern good categories of spectra have a well-defined and homotopically well-behaved smash product that allows one to perform algebraic operations on spectra as if they were classical rings in what is known as ``Brave New Algebra.''

At the time of the original algebraic $K$-theory constructions, researchers did not have access to these ``good'' categories of spectra, but in light of modern constructions of these categories, we can require more of algebraic $K$-theory. 
An obvious desideratum is to construct spectra from categorical data in such a way as to produce these kind of multiplicative structures from ``multiplicative'' types of categorical input.  Indeed, there is a large body of work on ``multiplicative infinite loop space machines'' of this sort \cite{maypairings,May82multiplicative,May09construction}.

In this paper, we focus on the Segal and Waldhausen constructions. Segal's algebraic $K$-theory construction is lifted to a multiplicative construction in \cite{EM2006}.  Waldhausen's construction is lifted to a multiplicative construction in \cite{BM2010}. Both lifts use the language of  symmetric multicategories and symmetric multifunctors to encode their multiplicativity.

While the May--Thomason theorem allows one to compare the underlying ``additive'' spectra of any such constructions, the techniques there do not extend to considering this new multiplicative structure.  Such a comparison is much to be desired: spectra may enjoy several multiplicative structures and it is important to be able to identify when we have the same one. In this paper, we show that two multiplicative $K$-theory constructions produce the same multiplicative spectra when given comparable input.  Our technique is to produce a multinatural transformation between the multifunctorial versions of the Segal and Waldhausen algebraic $K$-theory constructions which lifts Waldhausen's initial comparison of these constructions.  A multinatural transformation of this sort is automatically a comparison of symmetric multifunctors, meaning that our comparison is compatible with with the symmetry of the smash product in spectra.

The result here is not only important in developing brave new algebra structures on spectra.  Essentially the same structures arise in studying \emph{stable categories:} categories in which the morphisms form spectra.  In such a category, the composition pairing behaves like a multiplication.  Because we work with multicategories, our comparison applies directly to spectrally-enriched categories constructed from these different versions of $K$-theory.

This result is new but perhaps not unanticipated: it is expected that algebraic $K$-theory in all its different guises should be the same construction in the strongest possible sense, including the sum total of its structure. There are several recent results along this general line in an infinity-categorical context.  \begin{itemize}
\item Work of Blumberg--Gepner--Tabuada \cite{BGT2016} provides a uniqueness result for algebraic $K$-theory as a functor from the infinity category of idempotent-complete stable infinity categories that includes uniqueness of the multiplicative structure. 

\item Gepner--Groth--Nikolaus \cite{GGN2015} show that ``group completion'' in the sense of constructing a connective spectrum from an $E_\infty$-space is universal as a functor between symmetric monoidal infinity categories. This gives them universal multiplicative versions of $K$-theory as a functor from the infinity category of symmetric monoidal categories or the infinity category of symmetric monoidal infinity categories to the infinity category of spectra. 

\item Barwick \cite{BarwickMultiplicative} also provides a universal multiplicative structure on algebraic $K$-theory, in this context thought of as a functor from the infinity category of Waldhausen infinity categories to the infinity category of spectra.
\end{itemize}
Our result differs in two key ways from this work, and is not directly comparable.  First, by working with multicategories, we provide an on-the-nose comparison of multiplicative $K$-theory constructions.  This means our comparison has the ability to identify strictly commutative ring spectra as well as make other strict comparisons; this includes composition in spectral categories as mentioned above.  This ability to make strict comparisons is necessary for ``change of enrichment'' type results.

 Second, we focus on $K$-theory as a construction on Waldhausen categories, rather than any of the classes of infinity categories.  Since the classical Waldhausen and Segal constructions apply only to more specific types of input, their uniqueness is not addressed by any of the uniqueness or universality results above. The present paper provides this kind of multiplicative comparison between the classical versions of algebraic $K$-theory.

\subsection{Statement of results and structure of this paper}

The main result of this work is the construction of a multinatural transformation between the multifunctorial versions of Segal's and Waldhausen's algebraic $K$-theory functors. Since these two constructions traditionally start with different types of input categories---symmetric monoidal categories in the first instance and Waldhausen categories in the second---we must first establish a symmetric multifunctor comparing these inputs.   Essentially, our main result is to establish the following diagram:
\begin{thrm}[See Theorem \ref{thrm:nattransmultifunctors}]\label{intromainthrm} There is a natural transformation of symmetric multifunctors making the following diagram commute:
\[
\xymatrix{\mcdewald \ar[dr]_K\ar[rr] \drrtwocell<\omit>{<-1>}&&\mcsmc\ar[dl]^K\\
&\Spec&}
\]
\end{thrm}
Here $\mcsmc$ is the symmetric multicategory of symmetric monoidal categories with strict unit, discussed in  \cref{sect:multifunctorsmcKtheory}, and $\mcdewald$ is the symmetric multicategory of Waldhausen categories ``with choice of wedges''; this is essentially a choice of pairwise wedge products necessary to define the horizontal functor.  The precise details are in \cref{sec:waldtosymm}.

Elmendorf--Mandell \cite{EM2006} shows that ring structures on spectra, modules over ring spectra and algebras over ring spectra can be encoded as symmetric multifunctors from certain small symmetric multicategories to the symmetric multicategory of spectra.  Hence  \cref{intromainthrm} provides a comparison between spectral algebraic objects arising from these two constructions.  This is detailed in Corollaries \ref{cor:operadalg} and \ref{cor:multicategalgebra}.

In \cref{thrm:comparisonisequivalence}, we show that this natural transformation is an equivalence on Waldhausen categories with split cofibrations.  This is the classical condition under which Waldhausen \cite{Wald} provides an equivalence between Segal and Waldhausen $K$-theory spectra.

The structure of the paper is as follows.  In Section \ref{sec:spectra}  we give a quick overview of symmetric spectra.   In Section \ref{sect:overviewktheory}, we  provide the basic definitions of Waldhausen and Segal-style $K$-theory.   This is an ahistorical treatment: we use the Blumberg--Mandell definition of iterated Waldhausen $K$-theory from \cite{BM2010} and a generalization of the Elmendorf--Mandell version of $K$-theory \cite{EM2006} that allows for the symmetric monoidal category inputs.  Although the point of these models is that they are multifunctorial, at this point, we develop only the basic definitions and defer the symmetric multifunctor structure until Sections \ref{sect:multifunctorWaldKtheory} and \ref{sect:multifunctorsmcKtheory}.  In Section \ref{sect:multicategories}, we provide a brief introduction to symmetric multicategories and symmetric multifunctors, including enriched versions. Section \ref{sect:Ecat} introduces $\Ecat$, an auxiliary category used in proving the main result, and establishes its relationship to symmetric spectra. In Section \ref{sec:waldtosymm} we build the symmetric multifunctor that takes Waldhausen categories to symmetric monoidal categories.  Section \ref{sect:comparison} is devoted to the proof of the main theorem and Section \ref{sect:equivalence} discusses the question of when the transformation is an equivalence and gives corollaries summarizing some of the key consequences of the main results.

%

\subsection*{Acknowledgments} 
Many thanks to Clark Barwick, Andrew Blumberg, Tony Elmendorf, Lars Hesselholt, Mona Merling and Inna Zakharevich for interesting and helpful conversations. The authors also thank the anonymous referee for their very careful reading of this paper and for several clarifying suggestions.

It is a pleasure to acknowledge the support of several institutions that helped make this research possible. The first author was partially supported by NSF DMS-1710534.  The second author was partially supported by the Simons Foundation Grant No.~359449, the Woodrow Wilson
Career Enhancement Fellowship, and NSF grant DMS-1709302.  The first author also thanks the Isaac Newton Institute  for Mathematical Sciences for support and hospitality during the program ``Homotopy Harnessing Higher Structures'' when some of the work on this paper was undertaken.  This work was supported by  EPSRC grant numbers EP/K032208/1 and EP/R014604/1.

\section{Preliminaries on symmetric spectra}\label{sec:spectra}

In this paper, we use  symmetric spectra in simplicial sets \cite{HSS00Symmetric} as our model of spectra because they are the most natural target category for the two $K$-theory constructions. We briefly review some of the key elements of the theory of symmetric spectra and establish some notation we use in the sequel.  For further details, we refer the reader to \cite{HSS00Symmetric, MMSS01Model, SchwedeSymSpecBook}.

\begin{defn}\label{defn:simplicial_circle}
 The standard model for the circle as a based simplicial set is  $S^1=\Delta[1]/\partial\Delta[1]$. Thus, $S^1_n=\fps{n}=\{0,1,\dots,n\}$,  where $0$ is the basepoint, and for a map $\beta\colon [n]\to [m]$ in $\Delta$, the corresponding simplicial map $\beta^* \colon \fps{m} \to \fps{n}$ in $S^1$ sends $i$ to $j$ if $\beta(j-1)< i \leq \beta(j)$, and it sends $i$ to $0$ if there is no such $j$. 
\end{defn} 

This description of the maps $\beta^*$ can be deduced from the following observation.
\begin{rem}\label{rem:f_simplicial_circle}
Let $\mathcal{F}$ be the category with objects the finite pointed sets $\fps{n}$ as above, and pointed maps. The simplicial circle can be identified with the composite
\[\Delta^\op \longrightarrow \mathcal{F} \longrightarrow \set_*,\]
where the first functor is as described above, and the second functor is the inclusion into based sets. Under the identification $\mathcal{F}\cong \Gamma^\op$, this first functor is the canonical functor $\Delta^\op\to\Gamma^\op$ (see \cite[Remark 9.1]{MMO}).
\end{rem}

\begin{defn}\label{defn:symspec}
 A \emph{symmetric spectrum} $X$ consists of
 \begin{enumerate}
  \item a based simplicial set $X_n$ with a left $\Sigma_n$-action for each $n\geq 0$,
  \item maps of based simplicial sets $\sigma_n \colon X_n \sma S^1 \longrightarrow X_{n+1}$ for all $n\geq 0 $,
 \end{enumerate} 
 satisfying certain $\Sigma_n$-equivariance axioms as spelled out in \cite[Definition 1.2.2]{HSS00Symmetric}. 
\end{defn}

\begin{defn}\label{defn:specmor}
 A \emph{morphism} $f\colon X \longrightarrow Y$ of symmetric spectra consists of maps $f_n\colon X_n \longrightarrow Y_n$ of based simplicial sets that are $\Sigma_n$-equivariant, and compatible with the structure maps in the sense that the diagram
 \[ \xymatrix{ X_n \sma S^1 \ar[r]^-{\sigma_n} \ar[d]_{f_n \sma \id} & X_{n+1}\ar[d]^{f_{n+1}}\\
Y_n \sma S^1 \ar[r]_-{\sigma_n} & Y_{n+1}
 }
 \]
 commutes for all $n\geq 0$.
\end{defn}

We denote the category of symmetric spectra by $\Spec$. 

The category of symmetric spectra is tensored and enriched over the category of pointed simplicial sets. For a symmetric spectrum $X$ and a pointed simplicial set $K$, the symmetric spectrum $K\sma X$ is defined by
\[(K\sma X)_n=K \sma X_n,\]
with structure maps inherited from $X$. For symmetric spectra $X$ and $Y$, the mapping simplicial set $\map(X,Y)$ is defined as
\[\map(X,Y)_n=\Spec(\Delta[n]_+\sma X,Y),\]
with face and degeneracy maps induced by precomposition of the corresponding maps in $\Delta[-]$. Note that the 0-simplices can be naturally identified with the set of spectrum maps $X\longrightarrow Y$.

The category $\Spec$ has a symmetric monoidal structure given by the smash product. As we will not need the details of the construction here, we will instead concentrate on its universal property.

\begin{defn}\label{defn:bilinear} A \emph{bilinear map} $f\colon (X , Y) \longrightarrow Z$ of symmetric spectra consists of $\Sigma_p\times \Sigma_q$-equivariant maps of based simplicial sets 
\[f_{p,q}\colon X_p \sma Y_q \longrightarrow Z_{p+q},\]
for all $p,q,\geq 0$. These maps must be compatible with the structure maps, i.e., the diagrams
 \[ \xymatrix{ X_p \sma Y_q \sma S^1 \ar[r]^-{\id \sma \sigma_q} \ar[dd]_{f_{p,q} \sma \id} & X_p\sma Y_{q+1}\ar[dd]^{f_{p,q+1}}\\
 \\
Z_{p+q} \sma S^1 \ar[r]_-{\sigma_{p+q}} & Z_{p+q+1}
 } \qquad \quad
  \xymatrix{ X_p \sma  S^1 \sma Y_q  \ar[r]^-{\sigma_p \sma \id} \ar[d] & X_{p+1}\sma Y_{q}\ar[d]^{f_{p+1,q}}\\
X_p  \sma Y_q \sma S^1 \ar[d]_{f_{p,q} \sma \id}& Z_{p+1+q} \ar[d]\\
Z_{p+q} \sma S^1 \ar[r]_-{\sigma_{p+q}} & Z_{p+q+1}
 }
 \]
must commute for all $p,q\geq 0$. The unlabeled maps in the right diagram are given by the twist and by the action of the block permutation in $\Sigma_{p+q+1}$ that leaves the first $p$ elements fixed and swaps the block of the last $q$ elements with the $p+1$st element. One can similarly define $k$-linear maps, where the input is given by $k$-tuples of symmetric spectra.
\end{defn}

The smash product $X\sma Y$ comes with a bilinear map $(X,Y) \longrightarrow X \sma Y$, satisfying that for all $Z$, composition with this map induces a bijection
\[\Spec(X \sma Y, Z) \xrightarrow{\cong} \mathrm{Bilin}((X,Y),Z).\]
There is a similar bijection for iterated smash products and $k$-linear maps.

We also require symmetric spectra in simplicial categories, or more precisely, in based simplicial categories. As explained in \cite[Definition 7.1]{EM2006}, one can define symmetric spectra in simplicial objects in $\V_*$, where $\V$ is any bicomplete cartesian closed category. In what follows we specialize to the case $\V=\cat$, the category of small categories. Let $\ast$ denote the trivial category with one object and one morphism, which is final in $\cat$.  

\begin{defn}\label{defn:cat*}
A \emph{based category} consists of a category $\cC$ together with a functor $\ast \to\cC$. Note that this amounts to choosing a base object $\ast \in \cC$.
A \emph{based functor} $F\colon \cC \to \cD$ is a functor that preserves base objects.  A \emph{based natural transformation} $\phi$ between based functors $F,G\colon \cC\to \cD$ must satisfy the condition that the component $\phi_{\ast}$ at the base object $\ast\in\cC$ is the identity map $\id_\ast$ on the base object $\ast\in\cD$.
\end{defn}

Let $\cat_*$ denote the category of based small categories and based functors. We will also denote by $\cat_*$ the 2-category that includes based natural transformations.

The category $\cat_*$ has a coproduct $\vee$ and smash product $\sma$ defined analogously to the wedge product and smash product of based spaces.  See \cite[Construction 4.19]{EM2} for the construction of smash products in the general setting of based objects in a symmetric monoidal category $\V$. 

 Let $\scat_*$ denote the category of simplicial objects in based categories; that is, of functors $\Delta^\op\to \cat_*$.  This category is tensored over based simplicial sets, with the tensor $\cC\sma K$ of $\cC\in \scat_*$ and a based simplicial set $K$ given by 
\[ (\cC\sma K)_n=\bigvee_{K_n\setminus{*}}\cC_n.\]
We can then adapt verbatim \cref{defn:symspec,defn:specmor,defn:bilinear} to define symmetric spectra in $\scat_*$ and their maps.  The category of symmetric spectra in $\scat_*$ is denoted $\Spec(\scat)$.

Since $\cat_*$ is a 2-category, and hence so is $\scat_*$, we can further define 2-cells between morphisms of symmetric spectra, and more generally, between multilinear maps. 

\begin{defn} Let $f,g\colon \cC \to \cD$ be morphisms of symmetric spectra in $\scat_*$. A 2-cell $\alpha$ between them consists of based natural transformations $\alpha_n \colon \cC_n \to \cD_n$ in $\scat_*$ such that 
 \[ \begin{matrix}\xymatrixrowsep{30pt}\xymatrix{ \cC_n \sma S^1 \ar[r]^-{\sigma_n} \dtwocell<7>^{\ \quad f_n\sma \id}_{g_n\sma \id\ \quad}{\alpha_n \sma \id} & \cC_{n+1}\ar[d]^{f_{n+1}}\\
\cD_n \sma S^1 \ar[r]_-{\sigma_n} & \cD_{n+1}
 }\end{matrix} \quad = \quad \begin{matrix}\xymatrix{ \cC_n \sma S^1 \ar[r]^-{\sigma_n} \ar[d]_{f_n \sma \id} & \cC_{n+1}\dtwocell<7>^{\quad f_{n+1}}_{g_{n+1}\quad}{\alpha_{n+1}}\\
\cD_n \sma S^1 \ar[r]_-{\sigma_n} & \cD_{n+1}
 }\end{matrix}
 \]
\end{defn}
We thus see that $\Spec(\scat)$ is enriched over $\cat$.

As a right adjoint, the nerve functor $N\colon \cat \to \sset$ preserves products, and it extends to a lax symmetric monoidal functor $N\colon (\cat_*, \sma) \to (\sset_*,\sma)$. We can further consider the composite
\[\scat_* \fto{N} \ssset_* \fto{diag} \sset_*\]
obtained by applying the nerve levelwise to a based simplicial category to obtain a bisimplicial set, and then taking the diagonal (or geometric realization) to obtain a simplicial set. By abuse of notation, we also denote this functor by $N$.
\begin{prop}
 The nerve functor $N\colon \scat_* \to \sset_*$ is lax symmetric monoidal when considering the monoidal structure given by applying $\sma$ levelwise. We thus obtain a lax symmetric monoidal functor (of categories enriched in $\set$)
\[N\colon \Spec(\scat) \rightarrow \Spec.\]
\end{prop}

Lax monoidality of $N$ further allows us to use the nerve to change enrichments.  By applying $N$ at the level of  morphisms, we obtain a simplicially enriched category from any category enriched in categories.  Specifically, we can view $\Spec(\scat)$ as enriched in simplicial sets by applying $N$ to its categorical enrichment.  One then checks that the following proposition holds.

\begin{prop}\label{prop:changetosimpenrich}
The functor $N\colon \Spec(\scat)\to \Spec$ is a lax symmetric monoidal functor of categories enriched in simplicial sets when $\Spec(\scat)$ is given a simplicial enrichment via the nerve as above.
\end{prop}

This proposition allows us to work with categorical enrichment throughout this paper. Indeed, while the $K$-theory constructions of \cref{sect:overviewktheory} typically land in $\Spec$, we define them as categorically-enriched constructions landing in $\Spec(\scat)$. By first changing to simplicial enrichment by applying $N$ at the level of morphisms and then applying \cref{prop:changetosimpenrich}, one can obtain the more familiar simplicially-enriched constructions landing in $\Spec$.

\section{The $K$-theory constructions}\label{sect:overviewktheory}

In this section, we give the basic definitions of the two $K$-theory constructions this paper compares.  The details of the ``multiplicative'' or ``multifunctorial'' aspects of the constructions are deferred to  \cref{sect:multifunctorWaldKtheory,sect:multifunctorsmcKtheory}; at this point, we simply define the spaces of the $K$-theory spectra produced by the two constructions.  This presentation of the material is certainly ahistorical: the reason for both of these particular versions of $K$-theory constructions is that they extend beautifully to multifunctorial constructions.  However, the separation of the basic construction from the multifunctoriality allows us to highlight the multiplicative framework used to compare these constructions and (hopefully) clarifies some of the inevitable technical points.

\subsection{Iterated Waldhausen $K$-theory}

The version of Waldhausen's $S_\bullet$ construction that we use is an ``iterated $S_\bullet$ construction'' introduced in  \cite{GH2006}, with more details found in \cite{BM2010} and \cite{Zakh}. It is equivalent to iterating Waldhausen's original construction \cite{Wald}, but produces a multisimplicial structure rather than a simplicial structure and thus can be viewed as a multifunctor, as detailed in the references above and in \cref{sect:multifunctorWaldKtheory}.

We first recall the definition of a Waldhausen category \cite{Wald}. 

Recall that a category is \emph{pointed} if it has a distinguished object $*$ which is both initial and final.
\begin{defn}\label{defn:Wald_cat}
 A \emph{Waldhausen category} $\cC$ is a pointed category together with subcategories $w\cC$ (with morphisms called \emph{weak equivalences} and denoted by~$\xto{\sim}$) and $c\cC$ (with morphisms called \emph{cofibrations}, denoted by~$\hrto$), satisfying the following axioms.
 \begin{enumerate}
 \item All isomorphisms are contained in $w\cC$ and $c\cC$.
 \item For any object $X$ in $\cC$, the unique map $* \to X$ is a cofibration.
 \item If $Y \hlto X \to Z$ is a diagram in $\cC$, the pushout $Y \amalg _ X Z$ exists, and the map $Z \to Y \amalg _ X Z$ is a cofibration.
 \item Given a diagram
 \[
 \xymatrix{
 Y\ar[d]_{\sim} & X\ar[r] \ar[d]^{\sim} \ar@{_{(}->}[l]& Z\ar[d]^{\sim}\\
 Y' & X'\ar[r] \ar@{_{(}->}[l] & Z'
 }
 \]
 in $\cC$, the induced map $Y \amalg _ X Z \fto{\sim} Y' \amalg _ {X'} Z'$ is a weak equivalence.
 \end{enumerate}
\end{defn}

 We now establish some notation and terminology.

\begin{notn}
Let $[m]$ be the poset $0\leq 1 \leq \dotsb \leq m$ thought of as a category.  Let $\Ar[m]$ denote the category of arrows in $[m]$. The objects of $\Ar[m]$ are given by pairs of numbers $ij$ where $0\leq i\leq j\leq m$.   Given a natural number $n$ and  an $n$-tuple of natural numbers $(m_1,\dotsc,m_n)$, we let
\[\Ar[m_1,\dotsc,m_n]=\Ar[m_1]\times \dotsb \times \Ar[m_n].\]
\end{notn}

\begin{observation}\label{obs:comultisimp}
Since the category of arrows is functorial, as $m$ varies, the categories $\Ar[m]$ form a cosimplicial category.  This implies that $\Ar[-,\dots,-]$ is a comultisimplicial category with as many cosimplicial directions as inputs.
\end{observation}

\begin{defn}[{\cite[Definition 2.1]{BM2010}}]\label{defn:cubicallycofibrant} Let $\cC$ be a Waldhausen category. A functor $F\colon [m_1]\times\dots\times[m_n]\to \cC$ is \emph{cubically cofibrant} if
\begin{enumerate}
\item Each map $F(i_1,\dotsc,i_n)\to F(j_1,\dotsc,j_n)$ is a cofibration;
\item in every square given by $1\leq r < s\leq n$, 
\[\xymatrix{ F(i_1,\dotsc, i_n) \ar[r] \ar[d]&F(i_1,\dotsc,i_r+1,\dotsc, i_n)\ar[d]\\
F(i_1, \dotsc, i_s+1, \dotsc, i_n)\ar[r] & F(i_1,\dotsc,i_r+1,\dotsc, i_s+1,\dotsc, i_n),
}
\]
the induced map from the pushout to the lower right entry is a cofibration;
\item the higher dimensional analog of (2) holds: for a $k$-dimensional subcube given by increasing $k$ of the indices by $1$, the map from the pushout of the subcube without the terminal vertex to that terminal vertex is a cofibration.
\end{enumerate}
\end{defn}

We can now define the $n$th level of the $K$-theory spectrum of the Waldhausen category $\cC$. For each $n$, we define a $n$-fold simplicial category $S^{(n)}_{\bullet,\dotsc,\bullet}\cC$.
\begin{defn}\label{BM:iteratedSdotdefn} 
Let $m_1,\dotsc,m_n$ be natural numbers. Let $S^{(n)}_{m_1,\dotsc,m_n}\cC$ be the category whose objects are functors $A\colon\Ar[m_1,\dotsc,m_n]\to \cC$ satisfying the following conditions:
\begin{enumerate}
\item $A_{i_1j_1,\dotsc,i_nj_n}=\ast$ whenever $i_k=j_k$ for some $1\leq k\leq n$;
\item the functor $C\colon [m_1]\times \dots \times [m_n] \to \cC$ given by $C(j_1,\dots,j_m)=A_{0,j_1;\dots;0,j_n}$ is cubically cofibrant as in \cref{defn:cubicallycofibrant};
\item for every object $(i_1j_1,\dotsc,i_nj_n)$ in $\Ar[m_1,\dotsc,m_n]$, every $1\leq k\leq n$ and every $j_k\leq r\leq m_k$ the diagram
\[\xymatrix{ 
A_{i_1j_1,\dotsc,i_kj_k,\dotsc i_nj_n}\ar[r]\ar[d] &A_{i_1j_1,\dotsc,i_kr,\dotsc,i_nj_n}\ar[d]\\
A_{i_1j_1,\dotsc,j_kj_k,\dotsc i_nj_n} \ar[r] &A_{i_1j_1,\dotsc,j_kr,\dotsc,i_nj_n}
}
\]
is a pushout square.
\end{enumerate}
Morphisms in $S^{(n)}_{m_1,\dotsc,m_n}\cC$ are natural transformations. 
\end{defn}

Waldhausen's $S_\bullet$ construction is the special case of this construction where $n=1$. The cubical cofibrancy condition implies the required cofibrancy conditions in this case. Note that the definition makes sense even when $n=0$. In this case we are working with empty tuples; thus $\Ar[\,\,]$ is the terminal category, all of the conditions on the functors are vacuous, and we get that $S^{(0)}\cC \cong \cC$.   The $n$-fold simplicial structure on $S^{(n)}_{\bullet,\dotsc,\bullet}\cC$ is induced by the comultisimplicial structure on $\Ar[-,\dots,-]$ of \cref{obs:comultisimp}.

\begin{prop}\label{prop:BMisiteratedSdot}
The $n$-fold simplicial category $S^{(n)}_{\bullet,\dotsc,\bullet}\cC$ is isomorphic to the $n$-fold iteration of Waldhausen's $S_\bullet$ construction on $\cC$.
\end{prop}

To see this, note that the objects in $S^{(n)}_{m_1,\dotsc,m_n}\cC$ are functors $\Ar[m_1]\times \dots \times \Ar[m_n] \to \cC$, while the objects in $S_{m_1}(\dotsb(S_{m_n}(\cC))\dotsi)$ are functors $\Ar[m_1]\to S_{m_2}(\dotsb(S_{m_n}(\cC))\dotsi)$. Iterated use of the product-hom adjunction in $\cat$ allows one to view the latter as the former. Via this identification, one translates the conditions on $S_\bullet$ to those of \cref{BM:iteratedSdotdefn} to get the isomorphism.

Consider the $n$-fold simplicial category $wS^{(n)}_{\bullet,\dots,\bullet}\cC$ given by restricting the morphisms to those natural transformations whose components are weak equivalences. We can obtain a simplicial category by taking the diagonal.

\begin{defn}[\cite{BM2010}]\label{defn:BMktheory}
 The $K$-theory spectrum of a Waldhausen category $\cC$ is defined to be the symmetric spectrum $\waldk\cC$ in $\Spec(\scat)$ with $n$-th level given by the simplicial category
\[\waldk\cC(n) = diag(wS^{(n)}_{\bullet,\dotsc,\bullet}\cC).\]
 The $\Sigma_n$ action on $\waldk\cC(n)$ is given by permuting the sequences $m_1,\dotsc,m_n$. 

Of course, to give a complete definition of the $K$-theory spectrum $\waldk\cC$, we should discuss the structure maps. These can be found in \cite[\S6]{Zakh}, but they will also follow from our discussion of the multicategorical version in \cref{sect:multifunctorWaldKtheory}. See, in particular,  Theorems \ref{thrm:multifunctorEcattoSpec} and \ref{thrm:waldkismultifunctor}.
\end{defn}

\subsection{The $K$-theory of a symmetric monoidal category}

We next turn to the $K$-theory of a symmetric monoidal category.  Our definition is a minor refinement of the $K$-theory functor defined by Elmendorf and Mandell \cite{EM2006}, which is itself an adaptation of Segal's classical construction of the $K$-theory of symmetric monoidal category \cite{segal}. Elmendorf and Mandell produce a multifunctorial construction of the $K$-theory of a permutative category, that is, of a symmetric monoidal category that is strictly associative and has a strict unit.  In what follows, we consider strictly unital symmetric monoidal categories: ones with strict unit but whose associativity natural transformation need not be strict.  Since any symmetric monoidal category can be rigidified to either a strictly unital symmetric monoidal category or to a permutative category, we do not change the effective generality of the construction, but it is more convenient for our comparison.   Our adaptation of the Elmendorf--Mandell construction simply requires keeping track of associativity isomorphisms.

The context of strictly unital symmetric monoidal categories may seem somewhat ad hoc at first, but these categories are a natural source for $K$-theory constructions.  The unit functions as a basepoint and asking for structure that preserves the unit strictly is akin to requiring based maps.

\begin{defn} Let $(\cC,\oplus)$ be a small symmetric monoidal category, with unit $e$, symmetry $\gamma$ and associator $\alpha$.  We say that $\cC$ has a \emph{strict unit} if the functors $e\oplus -\colon \cC \to \cC$ and $-\oplus e\colon \cC \to \cC$ are equal to the identity functor, and the left and right unitors are the identity natural transformations---i.e., the natural maps $e\oplus c\to c$ and $c\oplus e\to c$ are the identity maps.   Note that this implies that instances of the symmetry and associativity natural transformations involving the unit must be the identity.
\end{defn}

\begin{defn}\label{strictlyunitalstrongsymmmonoidalfunctor}A \emph{strictly unital strong symmetric monoidal functor} between strictly unital symmetric monoidal categories is a symmetric monoidal functor $F\colon \cC\to \cD$ that is \begin{itemize}
\item \emph{strong} in that the coherence natural transformation $\delta\colon F\circ\oplus_\cD\Rightarrow\oplus_\cC\circ F$ is a natural isomorphism, and
\item \emph{strictly unital} in that $F(e)=e$ and the components of $\delta$ involving the units are the identity.
\end{itemize}
\end{defn}

\begin{defn}\label{EMdefn} Let $\cC$ be a small strictly unital symmetric monoidal category. Let $M_1,\dotsc, M_n$ be finite based sets.  In what follows, we will denote by $\langle S \rangle$ the $n$-tuple $(S_1,\dots,S_n)$, where $S_i$ is a subset of $M_i$ not containing the basepoint. For $1\leq i \leq n$ and $T$ a basepoint-free subset of $M_i$, we let $\langle S\lceil_i T\rangle$ denote the $n$-tuple $(S_1,\dots,S_{i-1},T,S_{i+1},\dots,S_n)$. 

Define $\bar{\cC}(M_1,\dotsc, M_n)$ to be the category whose objects are systems \[\{ C,\rho\}=\{C_{\langle S\rangle},\rho_{\langle S\rangle,i, T, U}\}\] where
\begin{itemize}
\item $\langle S\rangle=(S_1,\dotsc,S_n)$ runs through all $n$-tuples of subsets $S_i\subset M_i$ such that $S_i$ does not contain the basepoint,
\item in $\rho_{\langle S\rangle, i, T, U}$, $i$ runs through $1,\dotsc, n$ and $T$ and $U$ run through subsets of $S_i$ with $T\cap U=\emptyset$ and $T\cup U= S_i$,
\item $C_{\langle S\rangle}$ is an object of $\cC$,
\item $\rho_{\langle S\rangle,i,T,U}$ is an isomorphism\footnote{This definition varies slightly from that of Elmendorf and Mandell in that they do not require $\rho$ to be invertible here.  However, there is a natural adjunction between the category where the $\rho$'s are isomorphisms and the category where the $\rho$'s are merely morphisms, and so upon geometric realization we obtain weakly equivalent $K$-theory spaces.}
\[ C_{\langle S\lceil_i T\rangle} \oplus C_{\langle S\lceil_i U\rangle}\longrightarrow C_{\langle S\rangle}\]
in $\cC$,
\end{itemize} 
such that certain axioms hold.  These axioms are those of \cite[Construction 4.4]{EM2006}, except that we adapt them to include the associativity isomorphisms, any of which is denoted by $\alpha$ below.  That is, we require:
\begin{enumerate}
\item (Pointedness 1) $C_{\langle S\rangle}=e$ if any of the sets $S_k$ is empty,
\item (Pointedness 2) $\rho_{\langle S\rangle,i,T, U}$ is the identity if any of the sets $S_k, T$ or $U$ is empty,
\item (Symmetry) for all $\rho_{\langle S\rangle,i,T, U}$, the following diagram commutes:
\[
\xymatrixcolsep{2cm}\xymatrix{ C_{\langle S\lceil_i T \rangle}\oplus C_{\langle S\lceil_i U\rangle} \ar[r]^-{\rho_{\langle S\rangle,i,T,U}}\ar[d]_\gamma & C_{\langle S\rangle}\ar@{=}[d]\\
C_{\langle S\lceil_i U \rangle}\oplus C_{\langle S\lceil_i T\rangle} \ar[r]_-{\rho_{\langle S\rangle,i, U,T}} & C_{\langle S\rangle}
}
\]
\item (Associativity)  for all $\langle S\rangle$, $i$, and $T, U, V\subset M_i$ with $T\cup U\cup V=S_i$ and $T$, $U$, $V$ pairwise disjoint, the following diagram commutes:
\[\xymatrixcolsep{2.8cm}\xymatrix{ \left(C_{\langle S\lceil_i T\rangle} \oplus C_{\langle S\lceil_i U\rangle} \right)\oplus C_{\langle S\lceil_i V\rangle} \ar[r]^-{\rho_{\langle S\lceil_i T\cup U \rangle,i,T,U}\oplus \id}\ar[d]_{\alpha}  & C_{\langle S\lceil_i T\cup U\rangle} \oplus C_{\langle S\lceil_i V\rangle}
\ar[dd]^{\rho_{\langle S \rangle, i, T\cup U, V}}\\
C_{\langle S\lceil_i T\rangle}\oplus \left( C_{\langle S\lceil_i U\rangle}\oplus C_{\langle S\lceil_i V\rangle}\right)\ar[d]_{\id \oplus \rho_{\langle S \lceil_i U\cup V\rangle, i, U, V}} \\
C_{\langle S\lceil_i T\rangle} \oplus C_{\langle S\lceil_i U\cup V\rangle}\ar[r]_-{\rho_{\langle S\rangle, i, T, U\cup V}}
 & C_{\langle S\rangle} }
\]
\item (Coherence of the $\rho$'s) for all $\langle S \rangle$, all $i\neq j$, and all $T$, $U$, $V$, $W$ with $T\cup U=S_i$ and $V \cup W=S_j$, the following diagram commutes:
\[
\def\objectstyle{\scriptstyle}
\def\labelstyle{\scriptstyle}
\xy
(0,0)*+{\left(C_{\langle S \lceil_i T \lceil_jV \rangle}\oplus (C_{\langle S \lceil_i U \lceil_j V \rangle}\oplus C_{\langle S \lceil_i T \lceil_j W\rangle})\!\right)\oplus C_{\langle S \lceil_i U\lceil_j W\rangle}}="A";
(0,20)*+{(C_{\langle S \lceil_i T \lceil_j V \rangle}\oplus C_{\langle S \lceil_i U \lceil_j V \rangle})\oplus (C_{\langle S \lceil_i T \lceil_j W \rangle}\oplus C_{\langle S \lceil_i U \lceil_j W \rangle})}="B";
(0,-20)*+{\left(C_{\langle S \lceil_i T \lceil_j V \rangle}\oplus (C_{\langle S \lceil_i T \lceil_j W \rangle}\oplus C_{\langle S \lceil_i U \lceil_j V \rangle})\!\right)\oplus C_{\langle S \lceil_i U \lceil_j W \rangle}}="C";
(0,-40)*+{(C_{\langle S \lceil_i T \lceil_j V \rangle}\oplus C_{\langle S \lceil_i T \lceil_j W\rangle})\oplus (C_{\langle S \lceil_i U \lceil_j V \rangle}\oplus C_{\langle S \lceil_i U \lceil_j W \rangle})}="D";
(80,20)*+{C_{\langle S \lceil_j V \rangle}\oplus C_{\langle S \lceil_j W \rangle}}="E";
(80,-40)*+{C_{\langle S \lceil_i T \rangle}\oplus C_{\langle S \lceil_i U \rangle}}="F";
(80,-10)*+{C_{\langle S\rangle}}="G";
{\ar^{\alpha}_{\cong} "A";"B"};
{\ar^-{\rho_{\lceil_j V, i, T, U}\oplus \rho_{ \lceil_j W,i, T, U} }"B";"E"};
{\ar^-{\rho_{\langle S\rangle,j,V, W}} "E";"G"};
{\ar_-{(\id\oplus \gamma) \oplus \id}^{\cong} "A";"C"};
{\ar_{\alpha}^{\cong} "C";"D"};
{\ar_-{\rho_{\lceil_i T, j, V, W}\oplus \rho_{\lceil_i U,j,V,W}} "D";"F"};
{\ar_-{\rho_{\langle S\rangle, i, T, U}} "F";"G"};
\endxy
\]
In the horizontal maps, $\rho_{\lceil_j V,i,T, U}$ is shorthand for $\rho_{\langle S\lceil_jV\rangle,i, T, U}$ and so forth.
\end{enumerate}

A morphism $f\colon \{ C, \rho\} \to \{C', \rho'\}$ is a system of morphisms $f_{\langle S\rangle}\colon C_{\langle S\rangle}\to C'_{\langle S\rangle}$ for all $\langle S\rangle$ such that $f_{\langle S\rangle}$ is the identity when any $S_k$ is empty and the morphisms $f_{\langle S\rangle}$ commute with the morphisms  $\rho_{\langle S\rangle, i, T,U}$ in the natural sense.
\end{defn}

The categories $\overline{\cC}(M_1,\dotsc,M_n)$ fit together to form a functor from $n$-tuples of finite based sets to categories.  Permuting the sets $M_1,\dotsc, M_n$ yields an isomorphic category, and this, together with the other functorial structures enjoyed by $\overline{\cC}$, allows Elmendorf and Mandell to make the following definition of a $K$-theory spectrum of $\cC$. Recall the simplicial model of the circle from \cref{defn:simplicial_circle}. Since $S^1$ is a functor from $\Delta^\op$ to $\mathcal{F}$ (see \cref{rem:f_simplicial_circle}), we can precompose $\overline{\cC}$ with $(S^1)^{\times n}$ to produce an $n$-fold based simplicial category $\overline{\cC}(\underbrace{S^1,\dotsc, S^1}_n)$.

\begin{rem}
The construction $\overline{C}(-,\dotsc,-)$ of \cref{EMdefn} is functorial with respect to strictly unital strong symmetric monoidal functors.  The requirement that the maps $\rho$ be isomorphisms means we must require that the coherences $\delta$ of a monoidal functor be natural isomorphisms to obtain this functoriality.
\end{rem}

\begin{defn}[Elmendorf--Mandell] \label{EMKtheorydefn} For a small symmetric monoidal category $\cC$, the symmetric spectrum $\EMk(\cC)$ is defined at level $0$ by $K(\cC)(0)=\overline{\cC}(S^0)$ and at level $n>0$ by the simplicial category
\[\EMk(\cC)(n)=diag(\overline{\cC}(\underbrace{S^1,\dotsc, S^1}_n))\] 
\end{defn}
  
Again, this definition of the $K$ theory spectrum is not complete without the structure maps, but we defer the definition of these to \cref{sect:multifunctorsmcKtheory}.

\begin{rem}\label{remark:EMktheoryusuallyforisos}
This definition of the $K$-theory spectrum $\EMk(\cC)$ requires only that $\cC$ be a symmetric monoidal category.  However, the construction is typically applied in the case where $\cC$ is a groupoid.  The value of passing to groupoids is illustrated on $\pi_0$:  In the general case, $\pi_0\EMk(\cC)$  is the group completion of the monoid $\pi_0 B(\cC)$, which is the set of objects in $\cC$ modulo the relation that identifies two objects that are connected by a string of morphisms. In the case where $\cC$ is a groupoid, we obtain the group completion of the monoid of isomorphism classes of its objects.
\end{rem}

\section{Multicategories}\label{sect:multicategories}

As mentioned in the introduction, the multiplicative structure on the $K$-theory constructions of \cref{sect:overviewktheory} is encoded by describing them as multifunctors, that is, as functors between multicategories.  In this section, we recall the definitions necessary to make these concepts precise.

 The multicategories of interest in this paper are symmetric multicategories enriched in $\cat$ and $\sset$. We thus give the general definition of a symmetric multicategory enriched over an arbitrary symmetric monoidal category $(\V,\otimes,I)$.
For a full account, see \cite[Chapter 11]{yau}.

\begin{defn}A \emph{symmetric multicategory enriched in $\V$}, denoted $\bM$, consists of a collection of objects, denoted $\Ob\bM$, and for each $k\geq 0$ and objects $a_1,\cdots a_k, b$, an object $\bM(a_1,\dots, a_k; b)$ of $\V$.   The morphism objects are related by composition maps (in $\V$)
\[
\xymatrix{
\makebox[.2\textwidth]{$\bM(b_1,\dots b_n;c)\otimes \bM(a_1^1,\dots , a_1^{k_1};b_1)\otimes \cdots \otimes \bM(a_n^1,\dots , a_n^{k_n};b_n)$}\ar[d]^-{\Gamma}\\
 \bM(a_1^1,\dots , a_1^{k_1}, \dots ,a_n^1,\dots , a_n^{k_n};c).
}
\]
For each object $a$, there is an identity map $I \to \bM(a;a)$. Given $\sigma \in \Sigma_k$, there is a map in $\V$
\[\sigma^\ast \colon \bM(a_1,\dots,a_k;b)\to \bM(a_{\sigma(1)},\dots,a_{\sigma(k)};b).\]
All of this data is subject to associativity, identity and equivariance conditions, which can be found in \cite[Def 2.1]{EM2006} or \cite[Def 11.2.1]{yau}
\end{defn}

\begin{rem}\label{rem:composekarymorphisms}
As is the case with enriched categories, a symmetric multicategory $\bM$ enriched in $\V$ has an underlying symmetric multicategory (enriched in sets). We call the morphisms in the underlying multicategory the $k$-ary morphisms of $\bM$.  The identity map $I \to \bM(a;a)$ corresponds to the identity 1-ary morphism. In the cases where $\V$ is $\cat$ and $\sset$, the $k$-ary morphisms in the multicategory are given by the objects in the morphism category and the 0-simplices in the morphism simplicial set, respectively. 

Composing with $k$-ary morphisms induces maps in $\V$ between the different morphism objects. For example, a 1-ary morphism $h$ from $b$ to $c$ induces a map 
\[h\circ - \colon \bM(a_1,\dots,a_k;b) \to \bM(a_1,\dots,a_k;c).\]
\end{rem}

\begin{notn}
When $\V=\cat$, the morphisms of  $\bM(a_1,\dots, a_k; b)$ are called  \emph{$k$-ary cells}.
 Following the usual 2-categorical convention, $k$-ary morphisms $f,g\colon (a_1,\dots,a_k)\to b$ and $k$-ary cells $\beta \colon f \Rightarrow g$ are depicted as
 \[
 \xymatrix{
**[l](a_1,\dots,a_k) \rtwocell<4>^{f}_{g}{\beta} & b.
}\]
Composition of morphisms will be denoted by $(g; f_1,\dots f_n)\mapsto g\circ (f_1,\dots, f_n)$; similarly for compositions of cells. 
Note that the composition map in \cref{rem:composekarymorphisms} corresponds to the 2-categorical notion of whiskering. For example, for the $k$-ary cell $\beta$ as above and a 1-ary morphism $h\colon b \to c$, we denote by $h\circ \beta$ the composite $\id_h \circ \beta \colon h\circ f \Rightarrow h\circ g$.
\end{notn}

\begin{defn}
Given multicategories $\bM$ and $\bN$, a \emph{symmetric multifunctor} $F\colon \bM \to \bN$ consists of an assignment on objects $F\colon \Ob\bM \to \Ob\bN$, and for each tuple of objects $a_1,\dots,a_k,b$ of $\bM$, a map in $\V$ between morphism objects
\[ F\colon \bM(a_1,\dots,a_k;b)\to \bN(F(a_1),\dots,F(a_k);F(b)),\]
compatible with identity, composition and the $\Sigma_k$-action.
\end{defn}

\begin{defn}\label{defn:multinaturaltrans}
 Given symmetric multifunctors $F,G \colon \bM \to \bN$, a \emph{multinatural transformation} $\varepsilon\colon F \Rightarrow G$ consists of, for each $a\in \Ob \bM$, a 1-ary morphism $\varepsilon_a\colon F(a) \to G(a)$ in $\bN$, such that for all tuples $a_1,\dots,a_k,b$ of objects in $\bM$, the diagram in $\V$
 \[
 \xymatrix{
 \bM(a_1,\dots,a_k;b) \ar[r]^-{G} \ar[d]_F & \bN(G(a_1),\dots,G(a_k);G(b))\ar[d]^{-\circ (\varepsilon_{a_1},\dots,\varepsilon_{a_k})}\\
 \bN(F(a_1),\dots,F(a_k);F(b))\ar[r]_-{\varepsilon_b\circ -} & \bN(F(a_1),\dots,F(a_k);G(b))
 }
 \]
 commutes.
 \end{defn}
\begin{rem}\label{symmetricmultnattrans}
Observe that there is no additional condition required for multinatural transformation between symmetric multifunctors to respect the symmetry.  The same phenomenon occurs for symmetric monoidal categories: a ``symmetric monoidal natural transformation'' is just a monoidal natural transformation between symmetric monoidal functors.
\end{rem}

\begin{ex}\label{monoidalismulti} A symmetric monoidal category $(\mathcal{C},\oplus)$ gives rise to a symmetric multicategory (enriched in sets).  The objects of the multicategory are the objects of $\mathcal{C}$, and the $k$-morphisms are given by morphisms 
\[\left(\dotsi(a_1\oplus a_2)\dotsi \right)\oplus a_k\to b\]
 in $\mathcal{C}.$
If moreover $\mathcal{C}$ is a symmetric monoidal category enriched in $\V$, we can similarly construct an associated symmetric multicategory enriched in $\V$. Several of the enriched symmetric multicategories used in this paper arise in this way.  In this case, we use the same notation for the symmetric monoidal category and the symmetric multicategory.

Note that if $\cC$ and $\cD$ are symmetric monoidal categories, a  lax symmetric monoidal functor $F\colon \cC \to \cD$ canonically gives rise to a symmetric multifunctor of the corresponding symmetric multicategories. 
\end{ex} 

The most important examples of this type are given by the categories $\Spec$ of symmetric spectra in simplicial sets  and $\Spec(\scat)$ of symmetric spectra in simplicial categories (see \cref{sec:spectra}). The former is a symmetric monoidal category enriched in simplicial sets, and as such, it gives rise to a symmetric multicategory enriched in simplicial sets as well.  The latter is a symmetric monoidal category enriched in $\cat$.\footnote{Using the same idea as in the case of simplicial sets, one can also enrich $\Spec(\scat)$ in $\scat$. We will not need this second enrichment.} In both cases, the $k$-ary morphisms are precisely the $k$-linear maps of \cref{defn:bilinear}.  

\begin{rem}In the case of a categorically enriched multicategory, there is an unfortunate clash of nomenclature: a ``2-morphism'' might refer either to a 2-ary 1-morphism, that is, an object of a hom category of the form $\bM(a_1,a_2;b)$, or to a 2-cell, that is, a morphism in a hom category. To mitigate this possibility of confusion, we will consistently refer to the former as ``2-ary morphisms'' and the later as ``cells'' or more generally ``$k$-ary cells.''
\end{rem}

\begin{rem}\label{itsalwayssymmetric}
In what follows, we  use the term ``multicategory'' to refer to symmetric multicategories with whatever enrichment is specified.  Similarly, the terms ``multifunctor'' and ``multinatural transformation'' are used for symmetric enriched versions.
\end{rem}

\section{$\E_\ast$ categories and their relation to spectra}\label{sect:Ecat}

We now construct a multicategory $\Ecat$ that serves as an intermediate multicategory for the $K$-theory constructions of \cref{BM:iteratedSdotdefn} and \cref{EMdefn}.  This category is similar to the category $\Gcat$ constructed in \S5 of \cite{EM2006}, which the reader should consult to find details we omit. In brief, the objects of $\Ecat$ and $\Gcat$ consist of ``based'' functors from categories $\E$ and $\cG$ into the category of small categories, and the multicategorical structure on both categories is defined in the same way.  The difference is that $\cG$ is obtained via the Grothendieck construction on a functor given by taking powers of $\F$, whereas $\E$ is obtained via the Grothendieck construction on a related functor given by taking powers of $\Delta^\op$.

\begin{notn}
If $\beta\colon [m]\to [n]$ is a map in $\Delta$, let $\hat{\beta}$ denote the corresponding map $\hat{\beta}\colon [n]\to [m]$ in $\Delta^\op$.
\end{notn}

Let $\Inj$ denote the skeletal version of the category of finite sets and injections. The objects are given by the sets $\ul{r}=\{1,\dots,r\}$, where $r$ is a nonnegative integer. The category $\Inj$ is permutative, with monoidal product given by disjoint union. Consider the functor $(\Delta^{\op})^\ast \colon \Inj \to \cat$ that sends $\ul{r}$ to the category $(\Delta^{\op})^r$, and sends the injection $q\colon \ul{r} \to \ul{s}$ to the functor $q_* \colon (\Delta^{\op})^r \to (\Delta^{\op})^s $ given by
\[q_* ([m_1],\dots,[m_r]) = ([m_{q^{-1}(1)}],\dots,[m_{q^{-1}(s)}])\]
on objects. Here, by convention, if $q^\inv(j)=\emptyset$ for some $j$, we set $[m_{q^{\inv}(j)}]=[1]$. On morphisms, $q_*$ takes an $r$-tuple of morphisms $(\hat{\beta}_1,\dots,\hat{\beta}_r)$ to the tuple $(\hat{\beta}_{q^\inv(1)},\dots,\hat{\beta}_{q^\inv(s)})$ where by convention we insert the identity on $[1]$ whenever $q^\inv(j)=\emptyset$.

Let $\E=\Inj\wreath(\Delta^\op)^\ast$ be the category obtained by applying the Grothendieck construction to the functor $(\Delta^\op)^\ast$.  Concretely, this means that objects of $\E$ are given by (possibly empty) tuples of objects of $\Delta^\op$. Following the convention of \cite{EM2006}, we use bold and angle brackets to denote such a tuple, say  $\langle\mb{m}\rangle=([m_1],\dots,[m_r])$.  A morphism from an $r$-tuple $\langle \mb{m}\rangle=([m_1],\dots,[m_r])$ to an $s$-tuple $\langle \mb{m'}\rangle=([m'_1],\dots,[m'_s])$ consists of a pair $(q,\langle\hat{\beta}\rangle)$ where $q\colon\ul{r}\to \ul{s}$ is a morphism in  $\Inj$ and $\langle\hat{\beta}\rangle$ is a morphism 
\[\langle\hat{\beta}\rangle\colon q_*\langle\mb{m}\rangle\to \langle\mb{m'}\rangle\]
in $(\Delta^\op)^s$.
The composition of morphisms $(q',\langle\hat{\beta}'\rangle)\circ(q,\langle\hat{\beta}\rangle)$ is the morphism $(q'\circ q, \langle\hat{\beta}'\rangle\circ \langle q'_*\hat{\beta}\rangle)$.

\begin{rem}\label{factorE}
 Let $\inc_r\colon \ul{r} \hookrightarrow \ul{r+1}$ be the standard inclusion that misses $r+1$. A morphism in $\Inj$ can be factored as a composite of repeated inclusions of this type and permutations $\sigma \colon \ul{s} \to \ul{s}$. A morphism $(q,\langle\hat{\beta}\rangle)$ in $\E$ can be factored as $(\id,\langle\hat{\beta}\rangle)\circ (q,\id)$. Thus, morphisms in $\E$ are generated by morphisms of the form $(\inc_r,\id)$, $(\sigma,\id)$ and $(\id,\langle\hat{\beta}\rangle)$.
\end{rem}

The category $\E$ has a permutative structure given by concatenation of tuples. Following \cite{EM2006}, we denote this monoidal structure by $\odot$. 
\begin{rem} The existence of the permutative structure follows from the following general categorical fact. If $\D$ is a symmetric monoidal category, and $F\colon \D \to \cat$ is a lax symmetric monoidal functor (with respect to cartesian product), then the category $\D \wreath F$ obtained by applying the Grothendieck construction to $F$ has a symmetric monoidal structure compatible with that of $\D$, and moreover, if the monoidal structure on $\D$ is strict (i.e., $\D$ is a permutative category), then so is the one for $\D \wreath F$.
\end{rem}

\begin{defn}[cf.\,{\cite[Definition 5.2, Proposition 5.3]{EM2006}}]\label{defn:Ecat}
 An $\E_\ast$-category consists of a functor $X\colon \E \to \cat_\ast$ such that   $X([m_1],\dots,[m_r])=\ast$ if $m_i=0$ for some $i=1,\dots,r$.
 
 A map $F\colon X \to Y$ of $\E_\ast$-categories is a natural transformation of functors $\E \to \cat_\ast$.

\end{defn}

We now describe the categorically-enriched multicategory $\Ecat$ of $\E_\ast$-cat\-egor\-ies. A $k$-ary morphism $F\colon (X_1,\dots,X_k) \to Y$ between $\E_\ast$-categories $X_1,\dots,X_k$ and $Y$ 
 is given by a natural transformation
\[
\xymatrixcolsep{10ex}\xymatrix{
\E^k \ar[r]^-{X_1\times \dots \times X_k} \ar[d]_{\odot} \drtwocell<\omit>{F} & \cat_\ast^k\ar[d]^{\times}\\
\E \ar[r]_-{Y} & \cat_\ast
}
\]
compatible with basepoints, in the sense that for any object $(\langle \mb{m}_1\rangle,\dots,\langle \mb{m}_k\rangle)$ in $\E^k$, the functor
\[F\colon X_1(\langle \mb{m}_1\rangle)\times\dots\times X_k(\langle \mb{m}_k\rangle) \to Y(\langle \mb{m}_1\rangle\odot\dots\odot\langle \mb{m}_k\rangle)\]
satisfies
\begin{itemize}
\item $F(x_1,\dots,x_k)=\ast$ if $x_i=\ast$ is the basepoint in the category $X(\langle\mb{m}_i\rangle)$ for some $i=1,\dots,k$, and
\item $F(f_1,\dots,f_k)=\id_\ast$ if $f_i=\id_\ast$ for some $i=1,\dots,k$.
\end{itemize}
This is equivalent to requiring that the map factor through the smash product
\[X_1(\langle \mb{m}_1\rangle)\sma\dots\sma X_k(\langle \mb{m}_k\rangle)\]
of based categories.

A $k$-ary cell between $k$-ary morphisms $F$ and $G$ is given by a modification $\phi$ (remembering that $\cat_\ast$ is really a 2-category) that is compatible with the basepoints. In practice, what this means is that for every object $(\langle\mb{m}_1\rangle,\dots\langle\mb{m}_k\rangle )$ in $\E^k$, there is a natural transformation of functors \[\xymatrixcolsep{-5pc}
\xymatrix{
**[l]X_1(\langle\mb{m}_1\rangle)\times \dots\times  X_k(\langle\mb{m}_k\rangle) \rtwocell^{F}_{G}{\phi} &**[r]Y(\langle\mb{m}_1\rangle\odot \dots \odot \langle\mb{m}_k\rangle),
}
\]
satisfying compatibility conditions with respect to maps in $\E^k$ and basepoints. The condition on basepoints states that $\phi_{x_1,\dots,x_k}=\id_\ast$ if $x_i=\ast$ for some $i$. Note that this is equivalent to saying that $\phi$ is a based natural transformation of based functors 
\[\xymatrixcolsep{-5pc}\xymatrix{
**[l]X_1(\langle\mb{m}_1\rangle)\sma \dots\sma  X_k(\langle\mb{m}_k\rangle) \rtwocell^{F}_{G}{\phi} &**[r]Y(\langle\mb{m}_1\rangle\odot \dots \odot \langle\mb{m}_k\rangle).
}
\]

\begin{prop}[cf. {\cite[Proposition 5.4]{EM2006}}]
$\E_\ast$-categories, together with the $k$-ary morphisms and $k$-ary cells described above, form a categorically enriched multicategory $\Ecat$.
\end{prop}

Following \cite[Construction 7.3]{EM2006}, we construct a multifunctor from $\Ecat$ to the multicategory of symmetric spectra in $\scat$. Given an $\E_\ast$-category $X$ and a natural number $p\geq 0$, consider the simplicial pointed category $IX(p)$ given by the composite
\[\xymatrix{
\Delta^\op \ar[r]^-{\diag} &( \Delta ^\op)^p \ar[r] & \E \ar[r]^-{X} & \cat_\ast,  
}
\]
where the unlabeled map is the inclusion of $(\Delta^\op)^p$ into $\E$ as the fiber over $\ul{p}\in\Inj$. Note that $IX(p)$ has a $\Sigma_p$-action induced by the action on $(\Delta^\op)^p$. There are structure functors
\[IX(p) \sma S^1 \rightarrow IX(p+1),\]
which at simplicial level $q$ are given by the map
\[\bigvee_{j\in S^1_q\setminus \ast} X(\underbrace{[q],\dots,[q]}_{p})\rightarrow X(\underbrace{[q],\dots,[q]}_{p+1})\]
that on the wedge summand labeled by $j\in S^1_q\setminus *=\{1,2, \dots q\}$ is induced by the map $(\inc_p,(\id,\dots,\id,\hat{\beta}^j))$ in $\E$, where $\beta^j\colon [q]\to [1]$ is the map in $\Delta$ that sends $0,\dots,j-1$ to $0$ and $j,\dots, q$ to $1$.

\begin{thm}[{cf.\,\cite[Theorem 7.4]{EM2006}}]\label{thrm:multifunctorEcattoSpec} 
The assignment $I$ extends to a categor\-ically-enrich\-ed multifunctor
\[I \colon \Ecat \rightarrow \Spec (\scat).\]
\end{thm}

\section{Waldhausen $K$-theory as a multifunctor}\label{sect:multifunctorWaldKtheory}

In this section, we provide the fully multifunctorial definition of the Waldhausen $K$-theory construction of  \cref{BM:iteratedSdotdefn}.  As mentioned there, this definition is essentially due to Geisser and Hesselholt \cite{GH2006} and Blumberg and Mandell \cite{BM2010}.  However, rather than directly defining a multifunctor from a multicategory of Waldhausen categories to the multicategory of spectra as in \cite{BM2010}, we factor their construction as a multifunctor from Waldhausen categories to $\Ecat$ followed by the multifunctor from $\Ecat$ to $\Spec(\scat)$ of  \cref{thrm:multifunctorEcattoSpec}.

We begin by describing the multicategory of Waldhausen categories, which was first described in \cite{BM2010}.  For more details, see \cite{Zakh}.  Since we have already defined Waldhausen categories in \cref{sect:overviewktheory}, we need only to define the multimorphisms.

\begin{defn} 
 Let $\cC$ and $\D$ be Waldhausen categories. A functor $F\colon \cC \to \D$ is \emph{exact} if it sends $*$ to $*$, and preserves weak equivalences, cofibrations and pushouts along cofibrations.
\end{defn}

\begin{defn}
Let $\A_1, \dots, \A_k$ and $\B$ be Waldhausen categories. A functor
\[ F \colon \A_1 \times \cdots \times \A_k \to \B \]
is said to be \emph{$k$-exact} (or \emph{multiexact} if we wish to omit reference to $k$) if the following conditions hold:
\begin{enumerate}
\item the functor $F$ is exact in each variable; 
\item given cofibrations $f_i\colon X_{i,0} \hrto X_{i,1}$ in $\A_i$ for all $i=1,\dots k$, the cube $F(f_1, \dots, f_k)$ is cubically cofibrant.  That is, the functor 
\[F(f_1,\dots,f_k)\colon [1]^{\times k}\to \B\]
 is cubically cofibrant in the sense of \cref{defn:cubicallycofibrant}.
\end{enumerate}

Note that a 1-exact morphism is just an exact functor. A 0-exact morphism from the empty sequence into $\B$ is the choice of an object in $\B$. We can think of this as a functor from the empty product satisfying  no extra conditions, since both conditions for $k$-exactness are vacuous when $k=0$.
\end{defn}

Given multiexact functors $F_i \colon \A_i^1 \times \cdots \times \A_i ^{k_i} \to \B_i$ for $i=1, \dots, n$, and $G\colon \B_1 \times \cdots \times \B_n \to \cC$, we define their multicomposition as
\[ G\circ (F_1 \times \cdots \times F_n) \colon \A_1^1 \times \cdots \A_1 ^{k_1} \times \cdots \times  \A_n^1 \times \cdots \times \A_n^{k_n} \to \cC.\]
By \cite[Proposition 4.7]{Zakh}, this composition is again multiexact.

\begin{notn} \label{multicatofWaldcats}
 Given Waldhausen categories $\A_1,\dots,\A_k,\B$, we denote by $\mcwald(\A_1,\dots,\A_k;\B)$ the category of $k$-exact functors and all natural transformations. These assemble together to form the categorically-enriched multicategory $\mcwald$ whose objects are small Waldhausen categories (see \cite[Proposition 4.7]{Zakh}).
\end{notn}

\begin{rem}
Observe that the multimorphisms in  $\mcwald$ are simply functors satisfying some extra conditions, as opposed to having extra structure.  This makes constructing such multimorphisms straightforward, although one then needs to check the conditions are satisfied.
\end{rem}

\begin{rem}
As proved in \cite{Zakh}, the category $\mcwald(\A_1,\dots,\A_k;\B)$ is itself a Waldhausen category and composition is a 2-exact functor, giving $\mcwald$ the structure of a \emph{closed} multicategory. We will not use this extra structure in the present paper.
\end{rem}

We now show that Waldhausen $K$-theory provides a multifunctor from the multicategory $\mcwald$ to the multicategory $\Ecat$.

\begin{prop}
 The iterated Waldhausen construction can be assembled to give a $\E_*$-category $S^{()}_{\bullet,\dots,\bullet}\cC$, sending an object $([m_1],\dots,[m_r])$ to $S^{(r)}_{m_1,\dots,m_r}\cC$.
\end{prop}

\begin{proof}
  The basepoint of $S^{(r)}_{m_1,\dots,m_r}\cC$ is given by the constant functor at the zero object $*$.   
  Recall from \cref{factorE} that any morphism in $\E$ can be factored as a composite of morphisms of the form $(\inc_r,\id)$, $(\sigma,\id)$, and $(\id, \hat{\beta})$, where $\inc_r\colon \ul{r} \to \ul{r+1}$ is the ordered inclusion, $\sigma$ is a permutation, and $\hat{\beta}$ is a morphism in $(\Delta^\op)^r$. The image of 
 \[(\inc_r,\id)\colon ([m_1],\dots,[m_r]) \to ([m_1],\dots,[m_r],[1])\]
  is given by the extension isomorphism $e$ defined below in \cref{extension}. The image of 
  \[(\sigma,\id)\colon ([m_1],\dots,[m_n]) \to ([m_{\sigma^{-1}(1)}],\dots,[m_{\sigma^{-1}(r)}])\] 
  is the map $S^{(r)}_{m_1,\dots,m_r}\cC\to S^{(r)}_{m_{\sigma^\inv(1)},\dots,m_{\sigma^\inv(r)}}$ induced by permuting the inputs according to $\sigma$. For maps of the form $(\id,\hat{\beta})$ we use the multisimplicial structure. It is routine to check that these assignments preserve basepoints and are compatible with composition in $\E$, thus we get a functor $\E \to \cat_\ast$.
  
Note that if $m_i=0$ for some $i$, then then only functor $A\colon \Ar[m_1,\dots,m_r] \to \cC$ satisfying the conditions of \cref{BM:iteratedSdotdefn} is the constant functor at the zero object; thus in this case, $S^{(r)}_{m_1,\dots,m_r}\cC=\ast$ as required. 
\end{proof}

\begin{lemma}\label{extension}
 Given an $n$-tuple $(m_1,\dots,m_n)$, there is an extension isomorphism
 \[e\colon S^{(n)}_{m_1,\dots,m_n}\cC \longrightarrow S^{(n+1)}_{m_1,\dots,m_n,1}\cC\]
 compatible with the multisimplicial structure.
\end{lemma}

\begin{proof}
 The map $e$ sends a functor $A\colon \Ar[m_1,\dots,m_n]\to \cC$ to the functor $e(A)$ given by
 \begin{align*}
 e(A)_{i_1j_1,\dots,i_nj_n,00}&=\ast\\
 e(A)_{i_1j_1,\dots,i_nj_n,01}&=A_{i_1j_1,\dots,i_nj_n}\\
 e(A)_{i_1j_1,\dots,i_nj_n,11}&=\ast,
 \end{align*}
 with maps given by those of $A$ and by the unique maps from and to the zero object. It is routine to check that if $A$ satisfies the conditions of \cref{BM:iteratedSdotdefn}, so does $e(A)$. 
 
 The inverse is given by restricting a functor $B\colon \Ar[m_1,\dots,m_n,1]\to\cC$ to the subfunctor given by fixing $01$ in the last coordinates.
\end{proof}

\begin{thm}\label{thrm:waldkismultifunctor}
The assignment above gives a categorically-enriched multifunctor
\[\sdot\colon \mcwald \longrightarrow \Ecat.\]
Restriction to subcategories of weak equivalences then yields a categorically-enriched multifunctor 
\[w\sdot \colon \mcwald \longrightarrow \Ecat.\] 
\end{thm} 

\begin{proof}
 Let $F\colon \A_1 \times\dots\times \A_k \to \B$ be a $k$-exact functor of Waldhausen categories. For an object $(\langle {\bf m}_1\rangle,\dots,\langle {\bf m}_k\rangle)$ in $\E^k$, with tuples of length $(r_1,\dots,r_k)$, let $\langle \mb{m}\rangle = \langle\mb{m}_1\rangle \odot \dots \odot \langle \mb{m}_k \rangle$ and $r=r_1+\dots +r_k$. We must construct a functor
 \[\sdot(F)\colon S^{(r_1)}_{\langle \mb{m}_1\rangle} \A_1 \times \dots \times S^{(r_k)}_{\langle \mb{m}_k\rangle} \A_k \to S^{(r)}_{\langle \mb{m} \rangle} \B.\]
 An object of the source is given by a $k$-tuple $(A^1,\dots,A^k)$, where $A^i$ is a functor
 \[\Ar[m^i_1,\dots,m^i_{r_i}] \rightarrow \A_i\]
 satisfying the conditions in \cref{BM:iteratedSdotdefn}. The map $\sdot(F)$ sends the tuple $(A^1,\dots,A^k)$ to the composite of the isomorphism
 \[\Ar[m^1_1,\dots, m^1_{r_1},\dots,m^k_1,\dots, m^k_{r_k}]\cong \Ar[m^1_1,\dots, m^1_{r_1}]\times\dots\times\Ar[m^k_1,\dots, m^k_{r_k}]
 \]
 with $F\circ (A^1\times \dots \times A^k)$. We leave to the reader to check that this composite satisfies the conditions of \cref{BM:iteratedSdotdefn}, giving an object of the category $S^{(r)}_{\langle \mathbf{m} \rangle} \B$. The assignment of $\sdot(F)$ on morphisms  is defined similarly. It is standard to check that this assignment is natural with respect to maps in $\E^k$ and that it satisfies the basepoint conditions for $k$-ary cells described in \cref{sect:Ecat}. By construction, $\sdot$ respects composition of multimorphisms and the symmetric group action, thus giving a symmetric multifunctor as wanted.

The definition of $\sdot$ also extends to 2-cells between $k$-ary morphisms.  Given a natural transformation $\phi\colon F\Rightarrow G$ between $k$-exact functors, the 2-cell $\sdot(\phi)$ is essentially defined by whiskering.  It is straightforward to check that $\sdot$ is thus a categorically-enriched multifunctor.
 \end{proof}
 
Inspection of the definitions then shows that we indeed have our desired factorization of the Waldhausen $K$-theory multifunctor.

 \begin{prop}[{cf.\ \cite[Section 2]{BM2010}}]\label{prop:waldkcompositetospectra}
 The $K$-theory of \cref{defn:BMktheory} factors as
 \[
 \xymatrixcolsep{1.4cm}\xymatrix{
 \waldk\colon \mcwald \ar[r]^-{w\sdot} & \Ecat \ar[r]^-{I} & \Spec(\scat).
 }
 \]
This is a multifunctor of categorically enriched multicategories.
\end{prop}

\section{$K$-theory of symmetric monoidal categories as a multifunctor}\label{sect:multifunctorsmcKtheory}

In this section, we explain how the $K$-theory of symmetric monoidal categories in \cref{EMKtheorydefn} forms a multifunctor.  Again, we emphasize that the material in this section is not really new: this definition is a minor variation on the $K$-theory multifunctor from \cite{EM2006}, and our exposition below draws heavily from this work.  Although the reader familiar with \cite{EM2006} will find much of this material familiar, we nevertheless give details in order to make the comparison in \cref{sect:comparison} explicit.

The two main differences between what follows and the construction in \cite{EM2006} are these: First, we are working with strictly unital symmetric monoidal categories, rather than permutative categories.  This is purely a matter of convenience to allow comparison with Waldhausen's $K$-theory without the need to strictify.  Second, the $K$-theory multifunctor in \cite{EM2006} from the multicategory of permutative categories to the multicategory of spectra is defined by passing through an intermediate multicategory $\Gcat$.  In order to make the comparison between Waldhausen $K$-theory and symmetric monoidal $K$-theory in \cref{sect:comparison}, we choose instead to use the intermediate multicategory $\Ecat$ of \cref{sect:Ecat}.  In fact, while this category is not named and described in \cite{EM2006}, this factorization is simply a change of perspective on theirs.

As in the previous section, we begin by defining the multicategory of strictly unital symmetric monoidal categories that is the source of the $K$-theory functor.   This definition is a minor extension of the definition of the multicategory of permutative categories from \cite{EM2006}---by passing from permutative categories to strictly unital symmetric monoidal categories, we are simply relaxing the associativity requirements.  The material in this section boils down to keeping track of the associativity isomorphisms and checking that the constructions of \cite{EM2006} go through in this more general case.

We first define $k$-linear functors.  These will be the $k$-ary maps in the multicategory of strictly unital symmetric monoidal categories. Observe that a $1$-linear functor is just a strictly unital strong symmetric monoidal functor, as in \cref{strictlyunitalstrongsymmmonoidalfunctor}.

\begin{defn}\label{defn:klinear}
Let $\cC_1,\dotsc, \cC_k$ and $\cD$ be small strictly unital symmetric monoidal categories.  A \emph{$k$-linear} functor $(\cC_1,\dotsc,\cC_k)\to \cD$ is a functor
\[F\colon \cC_1\times \dotsb \times \cC_k\to \cD \]
together with distributivity natural isomorphisms 
\[ \delta_i\colon F(c_1,\dotsc, c_i,\dotsc,c_k)\oplus F(c_1,\dotsc,c_i',\dotsc,c_k)\to F(c_1,\dotsc,c_i\oplus c_i',\dotsc,c_k)\]
for $1\leq i\leq k$.  The requirement that these natural transformations be natural isomorphisms ensures that $k$-linear maps are compatible with the requirement that the maps $\rho$ be isomorphisms in \cref{EMdefn}. 
 For convenience, we will suppress the variables that are fixed, hence writing
\[\delta_i\colon F(c_i)\oplus F(c'_i) \to F(c_i \oplus c'_i).\]
The functor $F$ and the transformations $\delta_i$ must satisfy unitality conditions:
\begin{itemize}
\item $F(c_1,\dotsc,c_k)=e$ if $c_i=e$ for some $i$,
\item $F(f_1,\dots,f_k)=\id_e$ if $f_i=\id_e$ for some $i$, and
\item $\delta_i=\id$ if either $c_i,c'_i$ or any of the other $c_j$'s is $e$.
\end{itemize}
The transformations $\delta_i$  must make the following diagrams commute. 
\begin{enumerate}
\item (Compatibility of $\delta_i$ and $\alpha$.)\[\xymatrixcolsep{2cm}\xymatrix{ \left(F(c_i)\oplus F(c_i')\right)\oplus F(c_i'') \ar[d]_-{\delta_i\oplus \id}\ar[r]^{\alpha} & F(c_i)\oplus \left( F(c_i')\oplus F(c_i'')\right)\ar[d]^{\id \oplus \delta_i}\\
F(c_i\oplus c_i')\oplus F(c_i'') \ar[d]_{\delta_i} & F(c_i)\oplus F(c_i'\oplus c_i'')\ar[d]^-{\delta_i}\\
F( (c_i\oplus c_i') \oplus c_i'') \ar[r]_{F(\alpha)} & F(c_i\oplus (c_i'\oplus c_i'')) }
\]
\item (Compatibility of $\delta_i$ and $\gamma$.) \[\xymatrix{ F(c_i)\oplus F(c_i') \ar[d]_{\delta_i} \ar[r]^\gamma  & F(c'_i)\oplus F(c_i) \ar[d]^{\delta_i} \\
F(c_i\oplus c_i') \ar[r]_{F(\gamma)} & F(c_i'\oplus c_i)} \]

\item (Compatibility of $\delta_i$ and $\delta_j$.)  For all $i< j$
\[
\def\objectstyle{\scriptstyle}
\def\labelstyle{\scriptstyle}
\xy
(0,0)*+{\left(F(c_i,c_j)\oplus (F(c_i',c_j)\oplus F(c_i,c_j'))\right)\oplus F(c_i',c_j')}="A";
(10,20)*+{(F(c_i,c_j)\oplus F(c_i',c_j))\oplus (F(c_i,c_j')\oplus F(c_i',c_j'))}="B";
(0,-20)*+{\left(F(c_i,c_j)\oplus (F(c_i,c_j')\oplus F(c_i',c_j))\right)\oplus F(c_i',c_j')}="C";
(10,-40)*+{(F(c_i,c_j)\oplus F(c_i,c'_j))\oplus (F(c'_i,c_j)\oplus F(c_i',c_j'))}="D";
(70,20)*+{F(c_i\oplus c_i',c_j)\oplus F(c_i\oplus c_i',c_j')}="E";
(70,-40)*+{F(c_i,c_j\oplus c_j')\oplus F(c_i',c_j\oplus c_j')}="F";
(80,-10)*+{F(c_i\oplus c_i',c_j\oplus c_j')}="G";
{\ar^{\alpha} "A";"B"};
{\ar^-{\delta_i\oplus \delta_i} "B";"E"};
{\ar^-{\delta_j} "E";"G"};
{\ar_-{(\id\oplus \gamma) \oplus \id} "A";"C"};
{\ar_{\alpha} "C";"D"};
{\ar_-{\delta_j\oplus \delta_j} "D";"F"};
{\ar_-{\delta_i} "F";"G"};
\endxy
\]
\end{enumerate}
Here $\alpha$ denotes a composite of instances of the associator, which by the coherence of symmetric monoidal categories is uniquely determined by the two parenthesizations.

A $k$-linear natural transformation between $k$-linear functors $F$ and $G$ is a natural transformation $\phi\colon F\to G$ commuting with all the $\delta_i$'s in the sense that 
\begin{equation}\label{eqn:map_of_klinear}
\xymatrix{ F(c_i)\oplus F(c_i')\ar[r]^{\delta_i^F}\ar[d]_{\phi\oplus\phi} & F(c_i\oplus c_i')\ar[d]^{\phi}\\
G(c_i)\oplus G(c_i') \ar[r]_{\delta_i^G} &G(c_i\oplus c_i')
}
\end{equation}
commutes for every $i$.  Additionally we require that $\phi(c_1,\dotsc, c_k)=\id_e$ whenever any of the $c_i$'s is $e$.
\end{defn}

We describe a multicategory whose objects are symmetric monoidal categories with strict unit. 
\begin{defn}
Let $\mcsmc$ be the categorically-enriched multicategory whose objects are strictly unital symmetric monoidal categories and whose categories of $k$-morphisms $(\cC_1,\dotsc,\cC_k)\to \cD$ are given by the categories of $k$-linear functors and $k$-linear natural transformations.  The $\Sigma_k$-action on $k$-linear functors is given by acting on the indices $1,\dotsc, k$.  If we have multilinear functors $F_i\colon (\B_{i1}, \dots,  \B_{ik_i})\to \cC_i$ for $1\leq i\leq n$ and $G\colon(\cC_1,\dots, \cC_n)\to \D$, their composite is the $k_1+\dots+k_n$-linear functor 
\[\Gamma(G;F_1,\dots,F_n)=G\circ (F_1\times \dots\times F_n)\]
with distributivity transformations $\delta_s$ given as follows.  
The tuple $(k_1,\dots,k_n)$ is a partition of the $k_1+\dots+k_n$ ``inputs'' of $G\circ(F_1\times\dots\times F_n)$ into $n$ sets, and we fix $i$ so that $s$ is in the $i$th set of this partition; that is, $k_1+\dots+k_{i-1}<s\leq k_1+\dots+k_i$.  Then let $j=s-(k_1+\dots+k_{i-1})$, so that $s$ is the $j$th element in this set of the partition.
Then $\delta_s$ is 
\[ G(F_i(b_{ij}))\oplus G(F_i(b'_{ij})) \xto{\delta_i^G} G(F_i(b_{ij})\oplus F_i(b'_{ij})) \xto{G(\delta^{F_i}_j)} G(F_i(b_{ij}\oplus b'_{ij})).\]
It is relatively straightforward to check that these distributivity transformations satisfy the required diagrams.  Notice that the unit conditions for both $\Gamma(G; F_1,\dotsc,F_n)$ and the $\delta_i$'s automatically hold.
\end{defn}

The raison d'\^etre for Elmendorf and Mandell's version of $K$-theory is to handle multiplicative structures by showing that their construction of $K$-theory is a multifunctor from the multicategory of permutative categories to the multicategory of symmetric spectra.  We need the analogous result about our slightly adapted construction.  

\begin{thrm}[cf. {\cite[Theorem 6.1]{EM2006}}] \label{thrm:EMmultifunctortoEcat}
The construction of \cref{EMdefn} extends to an enriched multifunctor 
\[ \mcsmc \to \Ecat\]
and thus we obtain an enriched multifunctor
\[ \EMk \colon \mcsmc \to \Ecat \to \Spec(\scat),\] 
which at the level of objects is given by the construction in \cref{EMKtheorydefn}.
\end{thrm}

\begin{proof}
This essentially follows from Theorem 6.1 of \cite{EM2006}.  There the authors prove that the construction in \cref{EMdefn} gives an enriched multifunctor from permutative categories to $\Gcat$, where $\cG$ is the category obtained by applying the Grothendieck construction to the functor $\F^*\colon \Inj\to \cat$ that sends $\ul{r}\in\Inj$ to $\F^r$.  Recall from \cref{rem:f_simplicial_circle} that the standard simplicial circle can be considered as a functor $S^1\colon \Delta^\op \to \F$, and thus induces a natural transformation ${S^1}^*\colon {\Delta^\op}^*\Rightarrow \F^*$ between the functors defining $\E$ and $\cG$.  This induces a basepoint-preserving functor $\E_*\to \cG_*$ and precomposition with this functor yields multifunctor $\Gcat\to\Ecat$.  The proof of \cite[Theorem 6.1]{EM2006} involves constructing an enriched multifunctor from $\Gcat$ into $\Spec(\scat)$; this multifunctor factors through $\Ecat$, although the authors don't make that explicit.

Key to the work in this paper is that making \cref{EMdefn} into a functor from permutative categories to $\Gcat$ requires only that the unit in the permutative category be strict and not that the associator also be strict.  Hence the proof of \cite[Theorem 6.1]{EM2006} extends to strictly unital symmetric monoidal categories.
\end{proof}

In order to make the proof of \cref{sect:comparison} explicit, we unpack the construction of the functor in \cref{thrm:EMmultifunctortoEcat}.  Let $\cC$ be a strictly unital symmetric monoidal category.  Then the functor $\ol{\cC}\colon \E\to \cat_\ast$ is defined as follows.  As we note in \cref{defn:simplicial_circle}, the simplicial circle $S^1$ sends $[m]\in\Delta^\op$ to the finite pointed set $\fps{m}=\{0,\dots,m\}$ with 0 as the basepoint. Thus  the functor $\ol{\cC}$ sends an object $\langle \mb{m}\rangle=([m_1],\dots,[m_n])$ in $\E$ to the category $\ol{\cC}(\fps{m_1},\dots,\fps{m_n})$.   Hence if any $m_i=0$, $\ol{\cC}(\langle \mb{m}\rangle)$ is the terminal category $\ast$.

The image of a map of the form $(\inc_r,\id)$ in $\E$ is given by an extension functor as in \cite[page 184]{EM2006}.  
\begin{lemma}\label{lemma:smcKextensionfunctor}
For finite sets $M_1,\dots, M_n$, there is an isomorphism of categories $e\colon \ol{\cC}(M_1,\dots,M_n)\to \ol{\cC}(M_1,\dots,M_n,\fps{1})$ defined by sending a system $\{C,\rho\}\in \ol{C}(M_1,\dots,M_n)$ to the system $\{ C^e,\rho^e\}$ where 
\begin{align*}
C^e_{\langle S_1,\dots, S_n,\{1\}\rangle}&=C_{\langle S_1,\dots,S_n\rangle},&
 \rho^e_{\langle S_1,\dots,S_n,\{1\}\rangle,i,T,U}&=\begin{cases}\rho_{\langle S_1,\dots,S_n\rangle,i,T,U} &i<n+1\\ \mathrm{id} & i=n+1\end{cases}\\
C^e_{\langle S_1,\dots, S_n,\emptyset\rangle}&=e, &\rho_{\langle S_1,\dots,S_n,\emptyset\rangle,i,T,U}&=\mathrm{id}.
\end{align*}
The inverse isomorphism sends $\{C,\rho\}\in \ol{\cC}(M_1,\dots,M_n,\fps{1})$ to the system given by dropping the $\{1\}$ from list $\langle S_1,\dots,S_n,\{1\}\rangle$.
\end{lemma}

The image of a map in $\E$ of the form $(\sigma,\id)$ for $\sigma\in \Sigma_r$ is simply given by permuting the sets in the tuple $(\fps{m_1},\dots,\fps{m_r})$.   The image of a map in $\E$ of the form $(\id, \langle\hat{\beta}\rangle)$ is given by the multisimplicial structure referenced just before \cref{EMKtheorydefn}. That is, if $\langle\hat{\beta}\rangle=(\hat{\beta}_1,\dots,\hat{\beta}_n)\colon \langle \mb{m}\rangle\to\langle\mb{m}'\rangle$ is a map in $(\Delta^\op)^n$ and $\beta^*_i$ are the images of the $\beta_i$ under $S^1$, then $\ol{\cC}(\id,\langle\hat{\beta}\rangle)\colon \ol{\cC}(\langle \mb{m}\rangle)\to \ol{\cC}(\langle \mb{m}'\rangle)$ is the map induced by the tuple $(\beta^*_1,\dots,\beta^*_n)$ of maps of finite pointed sets.

\begin{rem}
An alternate way of proving that \cref{EMKtheorydefn} actually gives a multifunctor  is to observe that the construction in \cref{EMKtheorydefn} is a special case of the the $K$-theory of a small pointed  multicategories as defined by \cite{EM2}.  Specifically, \cite{EM2} shows that one can embed symmetric monoidal categories with strict unit in the category of pointed multicategories, and then apply the $K$-theory functor of a pointed multicategory as defined in \cite{EM2}.  This coincides with the definition we give above.
\end{rem}

\section{From Waldhausen to symmetric monoidal categories}\label{sec:waldtosymm}

In order to compare $K$-theory constructions, we first have to arrange for the two constructions to have comparable input.  A Waldhausen structure induces a symmetric monoidal structure, but to make this precise, one needs a predetermined choice of wedges. Let $\Sq\cC$ denote the set of commutative squares in $\cC$. 

\begin{defn}\label{defn:adjective}
 A \emph{Waldhausen category with choice of wedges} $(\cC,\omega)$ consists of a Waldhausen category $\cC$ together with a function 
 \[\omega\colon \Ob\cC \times \Ob\cC \rightarrow \Sq\cC\]
 such that 
 \begin{enumerate}
 \item for all $X,Y\in \Ob\cC$, $\omega(X,Y)$ is a pushout diagram 
 \[\xymatrix{
 \ast \ar[r] \ar[d] & X\ar[d]\\
 Y \ar[r]& Z;
 }\]
 \item for all $X$, 
 \[\omega(X,\ast)=\begin{matrix}\xymatrix{
 \ast \ar[r] \ar[d] & X\ar[d]^{\id}\\
 \ast \ar[r] & X
 }
\end{matrix}
 \qquad \text{and} \qquad\   \omega(\ast,X)=\begin{matrix}\xymatrix{
 \ast \ar[r] \ar[d] & \ast\ar[d]\\
 X \ar[r]_{\id}& X.
 }\end{matrix}\]
 \end{enumerate}
 Given such an $\omega$, we will use the following notation for the data of $\omega(X,Y)$:
 \[\xymatrix{
 \ast \ar[r] \ar[d] & X\ar[d]^{\iota_1}\\
 Y \ar[r]_-{\iota_2}& X\wed Y.
 }\]
 Note that the maps $\iota_1$ and $\iota_2$ are cofibrations by axioms (2) and (3) of \cref{defn:Wald_cat}.
 
 Given a function $\omega$ we have unique maps $\pi_1\colon X\wed Y \to X$ and $\pi_2\colon X \wed Y \to Y$ given by the universal property of the pushout, as follows.
 
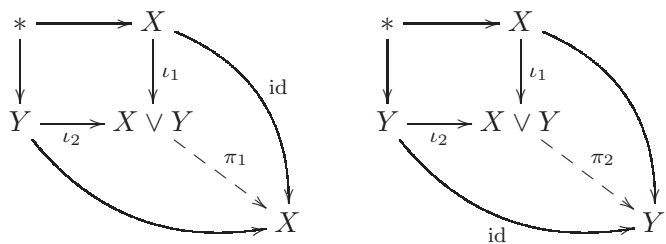
\[
 \xymatrix{
 \ast \ar[r] \ar[d] & X \ar[d]^{\iota_1}  \ar@/^1.5pc/[ddr]^{\id}\\
 Y \ar[r]_-{\iota_2} \ar@/_1.5pc/[drr] & X\wed Y \ar@{-->}[dr]^{\pi_1}\\
 & & X
 }
 \qquad
  \xymatrix{
 \ast \ar[r] \ar[d] & X \ar[d]^{\iota_1}  \ar@/^1.5pc/[ddr]\\
 Y \ar[r]_-{\iota_2} \ar@/_1.5pc/[drr]_{\id} & X\wed Y \ar@{-->}[dr]^{\pi_2}\\
 & & Y
 }
 \]
 The unlabeled maps are the unique maps that factor through $\ast$.
 
 A $k$-exact functor between  Waldhausen categories with choices of wedges is a $k$-exact functor of the underlying Waldhausen categories.
\end{defn}

 \begin{notn} The categorically-enriched multicategory $\mcdewald$ has as objects pairs $(\cC,\omega)$, of a small Waldhausen category and a choice of wedges on it, and categories of multimorphisms given by $k$-exact functors and natural transformations of such. No extra compatibility is imposed on the choices of wedges.
 \end{notn}

The following observation follows directly from the definition.

\begin{lemma}
 The underlying Waldhausen category of a Waldhausen category with choice of wedges gives the function on objects of a full and faithful multifunctor
 \[U \colon \mcdewald \rightarrow \mcwald.\]
\end{lemma}

\begin{rem} Note that given a Waldhausen category $\cC$, axioms (3) and (4) of \cref{defn:Wald_cat} imply that there exists at least one function $\omega$ making $(\cC,\omega)$ into a Waldhausen category with choice of wedges. In particular, this shows that $U$ is surjective on objects, and hence gives an equivalence of multicategories. Moreover, if $(\cC,\omega), (\cC,\omega')\in\mcdewald$ have the same underlying Waldhausen category $\cC$, then the identity functor $\cC \to \cC$ gives an isomorphism $(\cC,\omega)\to (\cC,\omega')$ in $\mcdewald$. Because all choices of $\omega$ are equivalent, we often omit $\omega$ from the notation.
\end{rem}

We now construct a multifunctor $\Lambda \colon \mcdewald \to \mcsmc$ which at the level of objects will send a pair $(\cC,\omega)$ to a strict symmetric monoidal structure on $\cC$.

\begin{const}\label{const:wald_to_smc}
Let $(\cC,\omega)$ be a Waldhausen category with a choice of wedges, as defined in \cref{defn:adjective}. Recall that for objects $X,Y$ in $\cC$, we denote by $X\wed Y$ the bottom right corner of the pushout diagram $\omega(X, Y)$. 

The universal property of the pushout implies this assignment extends to a functor $\wed \colon \cC \times \cC \to \cC$. Indeed, if $f\colon X \to X'$ and $g\colon Y \to Y'$ are morphisms in $\cC$, $f\wed g$ is defined to be the unique map $X\wed Y \to X'\wed Y'$ that makes the diagrams
\begin{equation}\label{iotanatural}
\begin{matrix}\xymatrix{
X \ar[r]^-{f} \ar[d]_{\iota_1} & X' \ar[d]^{\iota_1}\\
X \wed Y \ar[r]_-{f\wed g} & X'\wed Y'
}\end{matrix}
\quad \text{and} \quad
\begin{matrix}
\xymatrix{
Y \ar[r]^-{g} \ar[d]_{\iota_2} & Y' \ar[d]^{\iota_2}\\
X \wed Y \ar[r]_-{f\wed g} & X'\wed Y'.
}
\end{matrix}
\end{equation} 
commute. The uniqueness implies that this assignment respects composition and identities. Note that \cref{iotanatural} means that $\iota_1$ and $\iota_2$ are natural transformations. The universal property of the pushout can be used to prove that the projections $\pi_1$ and $\pi_2$ are natural as well, in the sense that the diagrams
\begin{equation}\label{pinatural}
\begin{matrix}\xymatrix{
X \wed Y \ar[r]^-{f\wed g} \ar[d]_{\pi_1}& X'\wed Y' \ar[d]^{\pi_1}\\
X \ar[r]_-{f}  & X' \\
}
\end{matrix}
\quad \text{and} \quad
\begin{matrix}\xymatrix{
X \wed Y \ar[r]^-{f\wed g} \ar[d]_{\pi_2}& X'\wed Y' \ar[d]^{\pi_2}\\
Y \ar[r]_-{g}  & Y' \\
}
\end{matrix}
\end{equation}
commute.

Given objects $X,Y$ in $\cC$, the universal property of pushouts implies the existence of an isomorphism $\gamma_{X,Y}\colon X \wed Y \to Y\wed X$, determined by the property of being the unique map such that $\gamma_{X,Y}\circ i_i=i_2$ and $\gamma_{X,Y}\circ i_2=i_1$. Universality implies that $\gamma$ is a natural transformation and that $\gamma_{Y,X}\circ \gamma_{X,Y}=\id_{X\wed Y}$.

Finally, for objects $X,Y,Z$ in $\cC$, universality implies the existence of an isomorphism $\alpha_{X,Y,Z} \colon (X \wed Y) \wed Z \to X \wed (Y \wed Z)$, which can be determined uniquely in terms of its interaction with the maps from $X$, $Y$ and $Z$ to the two pushouts. Uniqueness is used to prove that $\alpha$ is a natural transformation.
\end{const}

\begin{prop}
 Given a Waldhausen category with choice of wedges $(\cC, \omega)$, the functor $\wed$ and the natural transformations $\gamma$ and $\alpha$ of \cref{const:wald_to_smc} make $\Lambda(\cC,\omega)=(\cC,\wed,\gamma,\alpha,\ast)$ into a strictly unital symmetric monoidal category. This assignment on objects extends to give a categorically-enriched multifunctor
 \[\Lambda \colon \mcdewald \rightarrow \mcsmc.\]
\end{prop}

\begin{proof}
 By definition, $X\wed \ast = X$ and $i_1=\id_X$, and similarly, $\ast \wed X = X$ with $i_2=\id_X$. It follows that $\gamma_{X,\ast}=\gamma_{\ast,X}=\id_X$, and that $\alpha_{X,\ast,Y}=\id_{X\wed Y}$, showing that the associator and the symmetry interact appropriately with the strict unit. The pentagon axiom and the hexagon axiom follow from the uniqueness of the universal property of the pushout.
 
 Let $F\colon ((\A_1,\omega_1),\dots,(\A_k,\omega_k)) \to (\B,\omega)$ be a $k$-exact functor between Waldhausen categories with choices of wedges. This means in particular that $F$ is a functor $\A_1 \times \cdots \times \A_k \to \B$ that is exact in each variable. We will construct $\Lambda(F)=(F,\delta_i)$. Exactness in each variable implies that $F$ satisfies the first two unitality conditions of \cref{defn:klinear}, the second one following from the fact that $\ast$ is a zero object.
 
For $i=1,\dots,k$, since $F$ preserves pushouts along cofibrations on each variable, we have that $F(X_1, \dots , X_i\wed X'_i, \dots, X_k)$ is a pushout for the diagram
\[
\xymatrix{
\ast \ar[d] \ar[r] & F(X_1,\dots,X_i,\dots,X_k)\\
F(X_1,\dots,X'_i,\dots,X_k)
}
\]
Hence, there exists an isomorphism 
\[ F(X_1,\dots,X_i,\dots,X_k) \wed F(X_1,\dots,X'_i,\dots,X_k) \rightarrow F(X_1,\dots,X_i\wed X'_i,\dots, X_k)\]
with source our chosen pushout given by $\omega$: this is the map $\delta_i$. This map is the unique such map that is compatible with the chosen inclusions into the wedge. Careful inspection of the universal properties of the pushout will show that $\delta_i$ is natural, that the collection $\{\delta_i\}_{i=1,\dots,k}$ satisfies the axioms of \cref{defn:klinear}, and moreover, that this assignment is compatible with composition in the multicategories. 

Finally, given $k$-exact functors $F,G \colon ((\A_1,\omega_1),\dots,(\A_k,\omega_k)) \to (\B,\omega)$ and a natural transformation $\phi\colon F \Rightarrow G$, one can easily check that $\Lambda(\phi)=\phi$ satisfies \cref{eqn:map_of_klinear}, thus giving a map between $\Lambda(F)$ and $\Lambda(G)$.
\end{proof}

\begin{rem}
Note that the construction of $\Lambda(F)$ only made use of exactness in each variable. The conditions of cubical cofibrancy, which are necessary for the construction of $S^{()}_{\bullet,\dots,\bullet}$, are not needed here. Compare Barwick's infinity-categorical construction of the $K$-theory of a Waldhausen category in \cite{BarwickKthryhighercats}. 
\end{rem}

\begin{observation} Let $(\cC, \omega)$ be a Waldhausen category with choice of wedges. Then $w\Lambda(\cC,\omega)$, the category given by restricting morphisms in $\Lambda(\cC,\omega)$ to only the weak equivalences in $\cC$, is again a symmetric monoidal category.  This follows from the gluing lemma for weak equivalences in a Waldhausen category.
\end{observation}

\section{A multiplicative comparison of  Waldhausen and Segal $K$-theory}\label{sect:comparison}

We are now finally prepared to compare the multiplicative $K$-theory functors for Waldhausen and symmetric monoidal categories.

Specifically, we prove
\begin{thrm}\label{thrm:nattransmultifunctors} 
There is a multinatural transformation of categorically enriched multifunctors $\phi\colon \overline{(-)}\circ \Lambda \Rightarrow \sdot$  fitting into the following diagram:
\[\xymatrix{\mcdewald\ar[rr]^{\Lambda}\ar[dr]_-{S^{()}_{\bullet,\dotsc,\bullet}} \drrtwocell<\omit>{<-1>\phi}&& \mcsmc\ar[dl]^{\overline{(-)}}\\
&\Ecat&
}
\]
\end{thrm}
\begin{cor} Composition with the functor from $\Ecat$ to $\Spec(\scat)$ of \cref{thrm:multifunctorEcattoSpec} yields a multinatural transformation of categorically enriched multifunctors comparing Waldhausen $K$-theory and Segal $K$-theory:
\[\xymatrix{\mcdewald\ar[rr]^{\Lambda}\ar[dr]_-{S^{()}_{\bullet,\dotsc,\bullet}} \drrtwocell<\omit>{<-1>\phi}&& \mcsmc\ar[dl]^{\overline{(-)}}\\
&\Ecat\ar[d]_{I}&\\
&\Spec(\scat)
}
\]
\end{cor}
Changing enrichments along the nerve and applying \cref{prop:changetosimpenrich} then yields a multinatural transformation of simplicially enriched multifunctors that land in $\Spec$.

A multinatural transformation between multifunctors relating multicategories whose objects are structured categories is inherently a rather complicated gadget, so before proving such a thing exists, we begin by unpacking the structure involved. This is an exegesis of \cref{defn:multinaturaltrans} for the particular multicategories and multifunctors of \cref{thrm:nattransmultifunctors}.

For each Waldhausen category with choice of wedges $(\cC,\omega)$, the component $\phi_{\cC}$ of the multinatural transformation $\phi$ is a $1$-ary morphism in $\Ecat$
\[ \phi_{\cC}\colon \overline{\Lambda\cC}\to S^{()}_{\bullet, \dots,\bullet}\cC.\]
Since the objects in $\Ecat$ are functors from $\E$ to $\cat_*$, such a map is itself a natural transformation of functors 
\[ \xymatrixcolsep{1.5cm}\xymatrix{  \E \rtwocell<5>^{\overline{\Lambda\cC}}_{S^{()}_{\bullet,\dots,\bullet}}{\ \phi_{\cC}} &\cat}\]
that satisfies appropriate basepoint conditions.

At an object $\langle\mb{m}\rangle=([m_1],\dotsc,[m_n])$ in $\E$, the component of this natural transformation is a functor 
\[ \overline{\Lambda\cC}(\fps{m_1},\dots,\fps{m_n})\to S^{(n)}_{m_1,\dots,m_n}\cC.\]

This means that to prove \cref{thrm:nattransmultifunctors}, we must do the  following:
\begin{outline}\label{proofoutline}
\begin{enumerate}
\item For every $n\geq 0$ and every $n$-tuple of objects of $\Delta^\op$, we must construct a functor
\[(\phi_{\cC})_{\langle \mb{m}\rangle}\colon \overline{\Lambda\cC}(\fps{m_1},\dots,\fps{m_n})\to S^{(n)}_{m_1,\dots,m_n}\cC.\]
\item We must verify that these functors are natural with respect to maps in $\E$, showing that $\phi_{\cC}$ is a map of functors from $\E$ to $\cat$.
\item We must verify that the basepoint condition in \cref{defn:Ecat} holds, so that each $\phi_{\cC}$ is a $1$-ary morphism in the multicategory $\Ecat$.
\item We must verify that the maps $\phi_{\cC}$ fit together to form a natural transformation of multifunctors from $\mcdewald$ to $\Ecat$.  That is, the components of $\phi$ satisfy the condition on $k$-ary morphisms of \cref{defn:multinaturaltrans}.
\item We must verify that this natural transformation of multifunctors respects the $\cat$-enrichment enjoyed by these multicategories.  That is, the components of $\phi$ satisfy the condition on $k$-ary cells of \cref{defn:multinaturaltrans}.
\end{enumerate}
\end{outline}
We prove each of these statements in turn.  First, the heart of the matter: the actual construction of the functor in \cref{proofoutline} (1).  The underlying construction is essentially a multidimensional version of Waldhausen's construction in \cite[Section 1.8]{Wald}.  Before giving the full construction, let us recall the idea of  Waldhausen's 1-dimensional version.
\begin{recoll}\label{recollect1dim}
For a Waldhausen category $\cC$ and $[m]\in \Delta^\op$, Waldhausen \cite[Section 1.8]{Wald} discusses a functor from Segal's version of $\overline{\Lambda\cC}(\underline{m}_*)$ to $S_{m}\cC$. The idea of this functor is as follows.  An object in $\overline{\Lambda\cC}(\underline{m}_*)$ boils down to the data of an $m$-tuple $\{C_i\}$ of objects of $\cC$, one for each non-basepoint element of the set $\underline{m}_*$, together with the maps $\rho$ that indicate how to glue these together to get the objects corresponding to larger subsets of $\underline{m}_*$.  From this data, Waldhausen produces a  functor $A\colon \Ar[m]\to\cC$ which is usually displayed as a ``flag''
\[\xymatrix{
\ast\ar[r] & C_1 \ar[r]^-{\iota}\ar[d]& C_1\vee C_2 \ar[r]^-{\iota}\ar[d]_{\pi}& \ \cdots\ \ar[r]^-{\iota} & C_1\vee\dots\vee C_m\ar[d]_-{\pi}\\
& \ast \ar[r]& C_2 \ar[r]^-{\iota}\ar[d]_-{\pi}&\ \cdots\ \ar[r]^-{\iota} & C_2\vee \dots\vee C_m\ar[d]_-{\pi}\\
&&\ast&\ddots&\vdots\ar[d]_-{\pi}\\
&&&&C_m
}
\]
where the horizontal arrows are the inclusions into the corresponding wedge products, hence are cofibrations, and all the squares are pushout squares. Thus this gives an object of $S_{m}\cC$.
\end{recoll}

The following construction of the functor in  \cref{proofoutline} (1) is the multidimensional version of \cref{recollect1dim}
\begin{const}\label{constructBMobjfromEMobj}
Let $(\cC,\omega)$ be a Waldhausen category with wedges.  Fix $n$ and fix an $n$-tuple $\langle \mb{m}\rangle=([m_1],\dots,[m_n])$ of objects in $\Delta^\op$.  Let $\{C_{\langle S\rangle},\rho_{\langle S\rangle,i, T,U}\}$ be an object in $\overline{\Lambda \cC}(\fps{m_1},\dots,\fps{m_n})$.  From this object, we build a functor
\[A\colon \Ar[m_1,\dotsc, m_n]\to \cC.\]

For numbers $0\leq i\leq j\leq m$, let $(i,j]$ be the set $\{i+1, i+2, \dotsc, j\}$. This is a subset of $\fps{m}$ that does not contain the basepoint.  Note that if $i=j$, this is the empty set; if $i+1=j$, this is the one element set $\{j\}$.

Given an object $i_1j_1,\dots,i_nj_n$ in $\Ar[m_1,\dots,m_n]$, set the value of $A$ at this object to be
\[A_{i_1j_1,\dotsc, i_nj_n}=C_{\langle(i_1,j_1],\dotsc, (i_n,j_n]\rangle}.\]

Morphisms in $\Ar[m_1]\times\dotsb\times\Ar[m_n]$ can be factored as composites of morphisms in each component.  We explicitly define the value of the functor $A\colon \Ar[m_1,\dotsc,m_n]\to \cC$ on a morphism $i_kj_k\to i_k'j_k'$ in component $k$ to be the following composite.  For clarity, we only indicate the set in the $k$th component in the definition below, so that $C_{(i_k,j_k]}$ is shorthand for $C_{\langle(i_1,j_1],\dotsc,(i_n,j_m]\rangle}$.
\[\xymatrixcolsep{1pc}
\xymatrix{ C_{(i_k,j_k]}\ar[r]^-{\iota_1} & C_{(i_k,j_k]}\vee C_{(j_k,j_k']} \ar[r]^-{\rho}& C_{(i_k,j_k']}\ar[d]_{\rho^\inv} \\ && C_{(i_k,i_k']}\vee C_{(i_k',j_k']}\ar[d]_{\pi_2}\\ & & C_{(i_k',j_k']}
}
\]
As in \cref{const:wald_to_smc}, the horizontal map $\iota_1$ is the inclusion of a wedge summand and the vertical map $\pi_2$ is the projection, which exist by the universal property of pushout over the zero object.  The maps $\rho$ and $\rho^\inv$ are particular cases of the maps $\rho_{\langle S\rangle, i, T, U}$ that are part of the structure of an object of $\overline{\Lambda\cC}(\fps{m_1},\dotsc,\fps{m_n})$; we have omitted the indices because they are deducible from context, but it is worth noting that the map labeled $\rho$ and the map labeled $\rho^\inv$ have different indices and are in general not inverses of each other.

This defines $A$ on all morphisms as long as we can show that different factorizations of a morphism in $\Ar[m_1,\dotsc,m_n]$ into component morphisms yield the same definition of $A$ on the composite.  In other words, we must check that $A$ applied to a map in component $k$ and then a map in component $l$ is the same as $A$ applied to the map in component $l$ and then the map in component $k$.  The commutativity of the necessary diagram
\[\xymatrixcolsep{1pc}
\xymatrix{
C_{\langle(i_1,j_1],\dotsc,(i_n,j_n]\rangle} \ar[r]\ar[d]& C_{\langle(i_1,j_1],\dotsc,(i_k',j_k'],\dotsc,(i_n,j_n]\rangle} \ar[d]\\
C_{\langle(i_1,j_1],\dotsc,(i_l',j_l'],\dotsc,(i_n,j_n]\rangle}\ar[r] & C_{\langle(i_1,j_1],\dotsc,(i_k',j_k'],\dotsc,(i_l',j_l'],\dotsc, (i_n,j_n]\rangle}
}
\]
follows from the coherence condition (5) of \cref{EMdefn}, which relates the maps $\rho$ for varying indices $k$. 
\end{const}

\begin{prop}  The functor $A\colon \Ar[m_1,\dotsc,m_n]\to \cC$ defined in \cref{constructBMobjfromEMobj} is an object of the category $S^{(n)}_{m_1,\dotsc,m_n}\cC$.
\end{prop}

\begin{proof}
We must check that $A$ satisfies the three conditions of \cref{BM:iteratedSdotdefn}.  First, if there is some index $k$ so that $i_k=j_k$, the set $(i_k,j_k]$ is empty. Hence $C_{\langle(i_1,j_1],\dotsc,(i_k,j_k],\dotsc,(i_n,j_n]\rangle}$ is the zero object by axiom (1) of \cref{EMdefn}.

We next show the cubical cofibrancy condition of \cref{defn:cubicallycofibrant} holds for the subfunctor $A_{0j_1,\dotsc,0j_n}$.  For all $i_1\leq j_n, \dotsc, i_n\leq j_n$, the map
\[ A_{0i_1,\dotsc,0i_n}\to A_{0j_1,\dotsc,0j_n}\]
is a composite of maps of the form
\[ C_{(0,i_k]}\xto{\iota_1} C_{(0,i_k]}\vee C_{(i_k,j_k]} \xto{\rho} C_{(0,j_k]}. \]
The first map here is the pushout along the cofibration $\ast\to C_{(i_k,j_k]}$ and the second map is by definition an isomorphism; hence their composite is a cofibration.  This proves condition (1) of cubical cofibrancy.

Condition (2) of cubical cofibrancy requires that for all $1\leq k< l\leq n$, the induced map from the pushout of the diagram
\begin{equation}\label{diagramforcubicalcofibpushouts}
\xymatrix{  A_{0i_1,\dotsc,0i_n}\ar[d]\ar[r]& A_{0i_1,\dotsc,0(i_k+1),\dotsc,0i_n}\\ A_{0i_1,\dotsc,0(i_l+1),\dotsc, 0i_n} &&}
\end{equation}
to $A_{0i_1,\dotsc,0(i_k+1),\dotsc, 0(i_l+1),\dotsc,0i_n}$
be a cofibration. 
The structure isomorphisms $\rho$ of an object of $\overline{\Lambda\cC}(\fps{m_1},\dotsc,\fps{m_n})$ allow us to ``split off'' the last element of the sets $(,i_k+1]$ and $(0,i_l+1]$ and thus give an isomorphism between \cref{diagramforcubicalcofibpushouts} and the diagram
\[\xymatrix{C_{\langle S\rangle}\ar[r]^{\iota_1}\ar[d]_{\iota_1} &C_{\langle S\rangle}\vee C_{\langle S\lceil_k\{i_k+1\}\rangle}\\
{\hphantom{{}_{\lceil_l\{i_l+l\,\}}}C_{\langle S\rangle}\vee C_{\langle S\lceil_l\{i_l+1\}\rangle}}
}
\]
where $\langle S\rangle$ denotes the tuple $\langle (0,i_1],\dots,(0,i_n]\rangle$.
This is a diagram of the form $A\vee B_1\leftarrow A\to A\vee B_2$, 
so by inspection, a pushout of this diagram is 
\[ \left(C_{\langle S\rangle}\vee C_{\langle S\lceil_k\{i_k+1\}\rangle}\right)\vee C_{\langle S\lceil_l\{i_l+1\}\rangle}.\]
The structure isomorphisms $\rho$ combine to give an isomorphism between $C_{\langle S\lceil_k(0,i_k+1]\lceil_l(0,i_l+1]\rangle}$ and 
\[C_{\langle S\rangle}\vee
 C_{\langle S\lceil_k\{i_k+1\}\rangle}\vee C_{\langle S\lceil_l\{i_l+1\}\rangle}\vee C_{\langle S \lceil_k\{i_k+1\}\lceil_l\{i_l+1\}\rangle}.\]
The compatibility of the various $\rho$'s and the definition of the inducing map mean that induced map from the pushout of \cref{diagramforcubicalcofibpushouts} to this object is, up to isomorphism, the inclusion of a wedge summand, and therefore a cofibration.  A similar argument for higher dimensional cubes shows that the last condition of \cref{defn:cubicallycofibrant} also holds.

To complete the proof that \cref{constructBMobjfromEMobj} produces an object of $S^{(n)}_{m_1,\dotsc,m_n}\cC$, we observe that condition (3) of \cref{BM:iteratedSdotdefn} holds. For $0\leq i_k\leq j_k\leq r\leq m_k$, unpacking the definitions in the diagram of the form
\[\xymatrix{ 
A_{i_1j_1,\dotsc,i_kj_k,\dotsc i_nj_n}\ar[r]\ar[d] &A_{i_1j_1,\dotsc,i_kr,\dotsc,i_nj_n}\ar[d]\\
A_{i_1j_1,\dotsc,j_kj_k,\dotsc i_nj_n} \ar[r] &A_{i_1j_1,\dotsc,j_kr,\dotsc,i_nj_n}
}
\]
shows that it is isomorphic to the diagram
\[
\xymatrix{
C_{(i_k,j_k]} \ar[r]^-{\iota_1}\ar[d]& C_{(i_k,j_k]}\vee C_{(j_k,r]}\ar[d]^{\pi_2} \\
\ast\ar[r] & C_{(j_k,r]}
}
\]
where as before we have only indicated the sets in the $k$th coordinate.
The universal property defining the wedge product $\vee$ shows that this is a pushout diagram. 
\end{proof}

We have thus proved that \cref{constructBMobjfromEMobj} produces an object of $S^{(n)}_{m_1\dots,m_n}\cC$ from an object of $\overline{\Lambda \cC}(\fps{m_1},\dots,\fps{m_n})$.  To complete \cref{proofoutline} (1), we must observe that this assignment is actually a functor.  A morphism 
\[f\colon \{C,\rho\}\to \{C',\rho'\}\]
in $\overline{\Lambda \cC}(\fps{m_1},\dots,\fps{m_n})$ consists of a system of maps $\{f_{\langle S\rangle}\colon C_{\langle S\rangle}\to C'_{\langle S\rangle}\}$ that commute appropriately with the $\rho$'s and $\rho'$'s.  Given such maps, we must produce a natural transformation of functors 
\[\xymatrixcolsep{0pc} \xymatrix{ **[l]\Ar[m_1,\dots,m_n] \rtwocell<4>^{A}_{A'} &\cC .}\]
At each object $i_1j_1,\dots,i_nj_n$, this transformation has component
\[ f_{\langle(i_1,j_1],\dots,(i_n,j_n]\rangle}\colon C_{\langle (i_1, j_1],\dots,(i_n,j_n]\rangle}\to C'_{\langle (i_1,j_1],\dots,(i_n,j_n]\rangle}.\]
The fact that these are compatible with maps in $\Ar[m_1,\dots,m_n]$ follows from the fact that the $f$'s are compatible with the $\rho$'s as well as the naturality of the inclusion maps $\iota_1$ and the projection $\pi_2$ as in \cref{iotanatural} and \cref{pinatural}.  Hence \cref{constructBMobjfromEMobj} produces a functor 
\[(\phi_{\cC})_{\langle \mb{m}\rangle}\colon\overline{\Lambda \cC}(\fps{m_1},\dots,\fps{m_n})\to S^{(n)}_{m_1,\dots,m_n}\cC.\]

The next step in \cref{proofoutline} is to show (2):  that the functors $(\phi_{\cC})_{\langle\mb{m}\rangle}$ are natural with respect to maps in $\E$.
\begin{prop}\label{prop:naturalityinE}
For each $(\cC,\omega)$ in $\mcdewald$, the functors $(\phi_{\cC})_{\langle \mb{m}\rangle}$ are the components of a natural transformation 
\[ \xymatrixcolsep{1.5cm}\xymatrix{  \E \rtwocell<5>^{\overline{\Lambda\cC}}_{S^{()}_{\bullet,\dots,\bullet}\cC}{\ \,\phi_{\cC}} &\cat}.\]
\end{prop}
\begin{proof} 
A morphism $([m_1],\dots,[m_r])\to ([n_1],\dots,[n_s])$ in $\E$ is given by a pair $(q,\langle\hat{\beta}\rangle)$ where $q\colon \underline{r}\to \underline{s}$ is a morphism in $\Inj$ and $\langle\hat{\beta}\rangle\colon q_*([m_1],\dots,[m_r])\to ([n_1],\dots,[n_s])$ is a morphism in $(\Delta^\op)^s$.  We must show that the diagram of categories
\begin{equation}\label{naturalityinE}
\xymatrixcolsep{2.5pc}
\xymatrix{\overline{\Lambda \cC}(\fps{m_1},\dots,\fps{m_r})\ar[r]^{\overline{\Lambda\cC}(q,\langle\hat{\beta}\rangle)}\ar[d]_{\phi_{\cC}} &  \overline{\Lambda \cC}(\fps{n_1},\dots,\fps{n_s})\ar[d]^{\phi_{\cC}}\\
S^{(r)}_{m_1,\dots,m_r}\cC \ar[r]_{S^{()}_{\bullet,\dots,\bullet}(q,\langle\hat{\beta}\rangle)}& S^{(s)}_{n_1,\dots,n_s}\cC
}
\end{equation} 
commutes. 

As mentioned in \cref{factorE}, any map $(q,\langle\hat{\beta}\rangle)$ in $\E$ factors as a composite of maps of three types: $(\inc_r,\id)$, where $\inc_r$ is the inclusion of $\underline{r}$ into $\underline{r+1}$ as the first $r$ elements; $(\sigma,\id)$ where $\sigma\colon \underline{r}\to \underline{r}$ is a permutation, and $(\id,\langle\hat{\beta}\rangle)$ where $\langle\hat{\beta}\rangle\colon [m_1]\times \dots\times [m_r]\to [n_1]\times \dots\times [n_r]$ is a morphism in $(\Delta^\op)^r$.
To show that the $\phi$'s are natural with respect to maps in $\E$, it suffices to consider each of these types of maps in $\E$ separately.

The simplest case to understand is that of a permutation $(\sigma,\id)$.  Since $\sigma$ acts by permuting the $r$-tuple both in the $\overline{\Lambda \cC}$ part and in the $S^{()}_{\bullet,\dots,\bullet}$ part, \cref{naturalityinE} commutes when $(q,\langle\hat{\beta}\rangle)=(\sigma,\id)$. More explicitly, we can view the $K$-theory constructions of \cref{prop:BMisiteratedSdot} and \cref{EMdefn}  as producing functions
\[ \Ob(\Delta^\op)^r \xto{\overline{\Lambda \cC}} \cat \text{\qquad and\qquad} \Ob(\Delta^\op)^r\xto{S^{(r)}_{\bullet,\dots,\bullet}\cC} \cat.\]
For either construction, the action of $\Sigma_r$ is given by first permuting the $r$ factors in $\Ob(\Delta^\op)^r$, from which one can readily verify the necessary commutativity.

We next consider a morphism of the form $(\inc_r,\id)$. By construction, the map across the top of Diagram \eqref{naturalityinE} is the extension functor of \cref{lemma:smcKextensionfunctor} and the map across the bottom of Diagram \eqref{naturalityinE} is the extension functor of \cref{extension}. A  straightforward check shows that Diagram \eqref{naturalityinE} commutes for $(\inc_r,\id)$.

Finally, consider a morphism of the form $(\id, \langle\hat{\beta}\rangle)$ where $\langle\hat{\beta}\rangle$ is a morphism in $(\Delta^\op)^r$.  Since $(\Delta^\op)^r$ is a product, $\langle\hat{\beta}\rangle=(\hat{\beta}_1,\dots,\hat{\beta}_r)$, and it suffices to consider the case where there is only one value of $i$ such that $\hat{\beta}_i$ is not the identity.

Let $\hat{\beta}\colon [m]\to [n]$ in $\Delta^\op$ be the opposite of $\beta\colon [n]\to [m]$ in $\Delta$.  Recall that the map $\Lambda \cC(\fps{m})\to \Lambda \cC(\fps{n})$ across the top of Diagram (\ref{naturalityinE}) is induced by the map of finite pointed sets $\fps{m}\to \fps{n}$ given by $\beta^*$ in $S^1$,  whereas the map along the bottom of the diagram is given by precomposing with the induced map $\Ar[\beta]\colon \Ar[n]\to \Ar[m]$ of $\beta$ on arrow categories.

Suppose $\{C,\rho\}$ is an object of $\Lambda \cC(\fps{m})$.   Going around the top and right maps of Diagram (\ref{naturalityinE}) sends $\{C,\rho\}$ to the functor $A_1\colon \Ar[n]\to \cC$ with $A_1(ij)=C_{(\beta^*)^\inv(i,j]}$, where again $\beta^*$ is the map in $S^1$.  As observed in  \cref{defn:simplicial_circle}, $\beta^*$ takes $s\in \fps{m}$ to the unique $t\in \fps{n}$ such that $\beta(t-1)<s\leq \beta(t)$.  Hence $(\beta^*)^\inv(i,j]$ is the set 
\[\{t\in [n]\mid \beta(i)<t\leq \beta(j)\}=(\beta(i),\beta(j)].\]
Unpacking the composite around the left and bottom maps of Diagram (\ref{naturalityinE}),  we see that this composite also sends $\{C,\rho\}$ to the functor $A_2\colon \Ar[n]\to \cC$ with value $A_2(ij)=C_{(\beta(i),\beta(j)]}$.

Therefore $A_1,A_2\colon \Ar[n]\to \cC$ are the same on objects, and tracing through the images of the morphisms shows that they are the same functor.  This shows that when $(q,\langle\hat{\beta}\rangle)$ is of the form $(\id,\langle\hat{\beta}\rangle)$, Diagram \eqref{naturalityinE} commutes at the level of objects; since the morphisms in $\overline{\Lambda \cC}(\fps{m_1},\dots,\fps{m_r})$ are systems of compatible morphisms, the diagram commutes on that level as well.

This completes the proof that the functors $\phi_C$ form a natural transformation  of  functors $\E\to \cat_\ast$ as desired.
\end{proof}

We next show step (3) of \cref{proofoutline}: that the natural transformations $\phi_{\cC}$ are actually $1$-ary morphisms in the multicategory $\Ecat$.  
\begin{prop} The natural transformations $\phi_{\cC}$  of \cref{prop:naturalityinE} satisfy the basepoint preservation conditions necessary to be  $1$-ary morphisms in $\Ecat$.  
\end{prop}

\begin{proof}
As discussed in \cref{sect:Ecat}, in order to be a $1$-ary morphism, for each $\langle\mb{m}\rangle=([m_1],\dots,[m_r])\in \E$, the functor $\phi_{\cC}\colon \ol{\Lambda \cC}(\langle\mb{m}\rangle)\to S^{(r)}_{m_1,\dots,m_r}\cC$ must satisfy an object- and a morphism-level basepoint condition.  On objects, we require that  the functor $\phi_{\cC}$ take the basepoint of the category $\ol{\Lambda \cC}(\langle\mb{m}\rangle)$ to the basepoint of the category $S^{(r)}_{m_1,\dots,m_r}\cC$.  The basepoint in $\ol{\Lambda \cC}(\langle\mb{m}\rangle)$ is the constant system at the unit object $\ast\in\Lambda \cC$, which is the zero object in $\cC$. The map $\phi_\mathcal{C}$ sends this object  to the constant functor at the zero object $\ast\in \cC$, which is the basepoint in $S^{(r)}_{m_1,\dots,m_r}\cC$.  Thus the object-level basepoint condition holds.
On morphisms, an easy check shows that $\phi_{\cC}$ takes the identity morphism on the basepoint to the identity morphism on the basepoint.
\end{proof}

Step (4) in \cref{proofoutline} is to show that the maps $\phi_C$ form a multinatural transformation of multifunctors $\mcdewald \to \Ecat$.  
\begin{prop}\label{prop:multinatural}
The maps $\phi_C$ are multinatural on $k$-ary morphisms.  That is, given a $k$-exact functor $F\colon\cC_1\times \dots \times \cC_k\to \D$ of Waldhausen categories, we have an equality of $k$-ary morphisms
\[\phi_{\D}\circ \ol{\Lambda F} =\sdot F\circ (\phi_{\cC_1},\dots,\phi_{\cC_k})\]
in $\Ecat(\ol{\Lambda \cC_1},\dots,\ol{\Lambda \cC_k};\sdot\D)$.
\end{prop}
\begin{proof}
A $k$-ary morphism in $\Ecat$ is a natural transformation satisfying extra conditions.  To prove two $k$-ary morphism are equal, it thus suffices to check that the two composite  natural transformations in the following diagram are equal:
\begin{equation}\label{diagram:multinaturality}
\xymatrixcolsep{2cm}\xymatrix{ \overline{\Lambda\cC_1}\times \dots\times \overline{\Lambda\cC_k} \ar[r]^-{\overline{\Lambda F}}\ar[d]_{\phi_{\cC_1}\times\dots\times \phi_{\cC_k}} & \overline{\Lambda\D}\ar[d]^{\phi_{\D}}\\
S^{()}_{\bullet,\dots,\bullet}\cC_1\times \dots\times S^{()}_{\bullet,\dots,\bullet}\cC_k \ar[r]_-{S^{()}_{\bullet,\dots,\bullet} F} &S^{()}_{\bullet,\dots,\bullet}\D
}
\end{equation}
To verify this, we simply check that the components of both composite natural transformations at an object of $\E^k$ are the same.

For $1\leq i\leq k$, let $\langle\mb{m}_i\rangle=([m_{i1}],\dots,[m_{ir_i}])$ be an object in $\E$. Let $\langle\mathbf{m}_1\odot\dots\odot\mathbf{m}_k\rangle\in \E$ be the concatenation of the $\langle\mathbf{m}_i\rangle$'s and let $r=\sum_{i}r_i$.  We must show that two composite functors are the same in the following diagram of categories:
\[
\xymatrix{\overline{\Lambda \cC_1}(\langle\mb{m}_1\rangle)\times \dots\times \overline{\Lambda \cC_k}(\langle\mb{m}_k\rangle) \ar[r]\ar[d]& \overline{\Lambda \D}(\langle\mb{m}_1\odot\dots\odot\mb{m}_k\rangle)\ar[d] \\
S^{(r_1)}_{m_{11},\dots,m_{1r_1}}\cC_1\times\dots\times S^{(r_k)}_{m_{k1},\dots,m_{kr_k}}\cC_k \ar[r]& S^{(r)}_{m_{11},\dots,m_{kr_k}}\D
}
\]
Here and in what follows we write $\overline{\Lambda \cC_1}(\langle \mb{m}_1\rangle)$ for $\overline{\Lambda\cC_1}(\fps{m_{11}},\dots,\fps{m_{1r_{1}}})$ and so on for space considerations.

Let $(\{C^1, \rho^1\},\dots,\{C^k,\rho^k\})$ be an object in $\overline{\Lambda \cC_1}(\langle\mb{m}_1\rangle)\times \dots\times \overline{\Lambda \cC_k}(\langle\mb{m}_k\rangle)$.  The image of this object under the top right functors in the diagram above is the object $A\colon \Ar[\mathbf{m}_1,\dots,\mathbf{m}_k]\to \cD$ in $S^{(r)}_{m_{11},\dots,m_{kr_k}}\D$ whose value at an $r$-tuple of arrows $(i_{11}j_{11},\dots,i_{kr_k}j_{kr_k})$ is $D_{\langle (i_{11},j_{11}],\dots,(i_{kr_k},j_{kr_k}]\rangle}$ where $\{D,\rho\}$ is the system in $\overline{\Lambda D}$ defined as follows:  Given $\langle T\rangle$, which is a concatenation of lists $\langle S_i\rangle$ of subsets of $\{1,\dots,m_{ir_i}\}$ for $i=1,\dots k$, define 
\[D_{\langle T\rangle}=F(C^1_{\langle S_1\rangle},\dots,C^k_{\langle S_k\rangle}).\]

The image of $(\{C^1, \rho^1\},\dots,\{C^k,\rho^k\})$ around the left and lower maps in the above diagram is the functor $A'\colon \Ar[\mathbf{m}_1,\dots,\mathbf{m}_k]\to \D$ given by the composite
\[ \Ar[\mathbf{m}_1]\times\dots\times \Ar[\mathbf{m}_k] \xto{A_1\times\dots\times A_k} \cC_1\times \dots\times \cC_k\xto{F}\D\]
where the functor $A_\ell\colon \Ar[\mathbf{m}_\ell]\to \cC_i$ evaluated at $(i_{\ell 1}j_{\ell 1},\dots,i_{\ell r_\ell}j_{\ell r_\ell})$ is the object $C^\ell_{\langle(i_{\ell 1},j_{\ell 1}],\dots,(i_{\ell r_\ell},j_{\ell r_\ell }]\rangle}$.
Hence the functors $A$ and $A'$ coincide on objects and it is similarly straightforward to see they are also equal on morphisms.  

A similar check also shows that the necessary diagram commutes on morphisms in the category $\overline{\Lambda \cC_1}(\langle\mb{m}_1\rangle)\times \dots\times \overline{\Lambda \cC_k}(\langle\mb{m}_k\rangle)$.

This shows that $\phi$ is  respects $k$-ary morphisms.
\end{proof}

The final step in \cref{proofoutline} is to show that this multinatural transformation preserves the categorical enrichments; that is, the $k$-ary cells.  Since the objects and morphisms in $\Ecat(X_1,\dots,X_k;Y)$ are natural transformations and modifications, respectively, we must prove the following:

\begin{prop}\label{prop:modificationswork}
Given a natural transformation $\mu$ between $k$-exact functors $F,G\in \mcdewald(\cC_1,\dots,\cC_k;\D)$, the modifications
\[\phi_{\D}\circ \overline{\Lambda \mu}\text{\quad and \quad} S^{()}_{\bullet,\dots,\bullet}\mu \circ(\phi_{\cC_1}\times\dots\times\phi_{\cC_k})\]
 in $\Ecat(\ol{\Lambda \cC_1},\dots,\ol{\Lambda\cC_k}; \sdot \D)$ are equal.
\end{prop}
\begin{proof}
To show that these modifications agree, we again just have to check that the components at objects in $\E^k$ agree.  Each of these components is a natural transformation whose components come from the components of the original natural transformation $\mu$.  Let $(\langle\mathbf{m}_1\rangle,\dots,\langle\mathbf{m}_k\rangle)$ be an object in $\E^k$. Then the component of the modification $\phi_D\circ \overline{\Lambda\mu}$ at $(\langle\mathbf{m}_1\rangle,\dots,\langle\mathbf{m}_k\rangle)$ is a natural transformation between functors
\[\xymatrixcolsep{-1pc}
\xymatrix{**[l]S^{(r_1)}_{\mathbf{m}_1} \cC_1\times\dots\times S^{(r_k)}_{\mathbf{m}_k} \cC_k \rtwocell<4>^{<1.5>\phi_\D\circ \overline{\Lambda F}}_{<1.5>\phi_\D\circ \overline{\Lambda G}} & **[r]S^{(r)}_{\mathbf{m}_1,\dots,\mathbf{m}_k}\D
}
\]
and the component of $\sdot\circ (\phi_{\cC_1}\times \dots\times\phi_{\cC_k})$ at $(\langle\mathbf{m}_1\rangle,\dots,\langle\mathbf{m}_k\rangle)$ is a natural transformation of functors
\[
\xymatrix{**[l]S^{(r_1)}_{\mathbf{m}_1} \cC_1\times\dots\times S^{(r_k)}_{\mathbf{m}_k} \cC_k \rtwocell<5>^{<1.5>S^{()}_{\bullet,\dots,\bullet}F\circ(\phi_{\cC_1}\times\dots\times\phi_{\cC_k}) }_{<1.5>S^{()}_{\bullet,\dots,\bullet}G\circ (\phi_{\cC_1}\times \phi_{\cC_k})} &  **[r] S^{(r)}_{\mathbf{m}_1,\dots,\mathbf{m}_k}\D.
}
\]
In both cases, the component of the natural transformation in question at  functors $A_i\colon\Ar[\mathbf{m}_i]\to \cC_i$ is ultimately given by composition with the components of the original natural transformation $\mu\colon F \Rightarrow G$.   Hence the two modifications agree.
\end{proof}

\section{Weak equivalences and the multinatural equivalence}
\label{sect:equivalence}

In Theorem \ref{thrm:nattransmultifunctors}, we constructed a multinatural transformation between the multifunctors $\EMk \circ \Lambda$ and $\waldk$, as functors from Waldhausen categories with choices of wedges to $\Ecat$. However, we have not yet taken into account the weak equivalences nor have we shown that this transformation is an equivalence in reasonable cases.  In this section, we remedy these omissions.

Let $\wecat$ be the category of categories-with-weak-equivalences, that is, the category of pairs $(\cC, w\cC)$ where $\cC$ is a category and $w\cC$ is a subcategory of weak equivalences in $\cC$.  The subcategory  $w\cC$ must at least contain all objects of $\cC$ and might additionally be required to satisfy other properties.  For our purposes, we only require that all isomorphisms are contained in $w\cC$.   Morphisms in $\wecat$ are required to be ``exact'' in the sense of sending weak equivalences to weak equivalences.

By neglect of structure, every Waldhausen category $\cC$ is an object in $\wecat$.  Moreover, as in \cref{BM:iteratedSdotdefn}, for each $([m_1],\dots,[m_r])$ in $\E$, $S^{(r)}_{m_1,\dots,m_r}\cC$ has a subcategory of weak equivalences that consists of natural transformations each of whose components is a weak equivalence.  This makes $S^{(r)}_{m_1,\dots,m_r}\cC$ an object of $\wecat$.  Furthermore, the morphisms in $\E_*$ induce exact functors between the categories $S^{(r)}_{m_1,\dots,m_r}\cC$, so in fact  \cref{thrm:waldkismultifunctor} can be improved to the statement that $\sdot$ is a multifunctor
\[\sdot\colon \mcwald \to \Ewecat.\]
Indeed, the second statement of \cref{thrm:waldkismultifunctor}, that restriction to subcategories of weak equivalences also yields a multifunctor to $\Ecat$, encodes this statement: for each $\cC$, restriction to the subcategories of weak equivalences is given by the composite
\[ \E_*\xto{\sdot\cC}\wecat \xto{w} \cat\]
where $w\colon \wecat\to \cat$ forgets down to the subcategory of weak equivalences.  Making such a composite requires that $\sdot\cC$ take morphisms in $\E$ to exact functors between categories with weak equivalences.

Now we turn to the case of the $K$-theory of a symmetric monoidal category.  Symmetric monoidal categories don't naturally come with a subcategory of weak equivalences (aside from the trivial choice of the \emph{core}, the wide subcategory of all isomorphisms).  This means that the natural landing point for $\EMk$ is $\Ecat$, rather than $\Ewecat$ as in the Waldhausen case. However, it's worth observing that in practice one actually wants to apply $\EMk$ to the core of a symmetric monoidal category, as in \cref{remark:EMktheoryusuallyforisos}, rather than an arbitrary symmetric monoidal category.

Consider $(\cC,\omega)\in \mcdewald$. As in \cref{sect:multifunctorWaldKtheory}, $\Lambda\cC$ is a symmetric monoidal category, but it is more than that.  $\Lambda\cC$ is in fact a symmetric monoidal category that comes equipped with a sub-symmetric monoidal category of weak equivalences $w\Lambda\cC$, given by the restricting to the weak equivalences in $\cC$.  The gluing axioms ensure that $w\Lambda\cC$ is again symmetric monoidal.  One can thus apply $\EMk$ to obtain the $\E_*$-category $\EMk(w\cC)$.  This is the functor whose value at an object $([m_1],\dots,[m_r])\in \E_*$ is the category $\overline{w\Lambda\cC}(\fps{m_1},\dots,\fps{m_r})$ of systems of objects in  $w\Lambda\cC$.

However, one can also obtain this $\E_*$-category as the restriction to weak equivalences of an $\E_*$-category-with-weak-equivalences $\EMk(\cC)$ as follows.  View $(\Lambda(\cC),w\Lambda\cC)$ as an object in $\mcsmcwe$, the multicategory of strictly unital symmetric monoidal categories with unit equipped with a sub-symmetric monoidal category of weak equivalences.  Now the composite $\EMk(\Lambda-)$ takes values not just in $\Ecat$ but in $\Ewecat$.  For each $([m_1],\dots,[m_r])\in \E_*$,   the category$\overline{\Lambda\cC}(\fps{m_1},\dots,\fps{m_r})$ has a natural subcategory of weak equivalences, namely, the maps of systems $\{f\}\colon \{C,\rho\}\to \{C', \rho'\}$ so that all the components $f_{\langle S\rangle}$ are weak equivalences. It is routine to verify again that morphisms in $\E_*$ induce exact functors.  We thus have an $\E_*$-category given by the composite
\begin{equation}\label{weakequivforsmcktheory} w\overline{\Lambda\cC}\colon\E_* \xto{\overline{\Lambda\cC}} \wecat \xto{w} \cat.\end{equation}

The key observation is that since all the maps $\rho$ appearing in the systems of $\overline{\Lambda\cC}(-)$ are isomorphisms and hence weak equivalences, for any object $\langle [m_1],\dots,[m_r]\rangle$ in $\E$ we have an equality
\[ w\overline{\Lambda\cC}(\fps{m_1},\dots,\fps{m_r})=\overline{w\Lambda\cC}(\fps{m_1},\dots,\fps{m_r}).\]
Hence, passing to the subcategory of weak equivalences after taking $K$-theory is the same as taking $K$-theory after passing to the subcategory of weak equivalences.  That is, the functor of \cref{weakequivforsmcktheory} is the $\E_*$-category $\EMk(w\Lambda\cC)$.

In summary, we see that in fact we have the following diagram at the level of multicategories enriched in sets:
\[\xymatrix{\mcdewald\ar[rr]^{\Lambda}\ar[dr]_-{\sdot} \drrtwocell<\omit>{<-1>\phi}&& \mcsmcwe\ar[dl]^{\overline{(-)}}\\
&\Ewecat\ar[d]_{w}&\\
&\Ecat\ar[d]_{\abs{\cdot}\circ N\circ \diag}\\
&\Spec
}
\]

One observes here that our transformation $\phi$ takes weak equivalences in $\overline{\Lambda\cC}(\fps{m_1},\dots,\fps{m_r})$ to weak equivalences in $S^{(r)}_{m_1,\dots,m_r}\cC$ by construction: if $f$ is a map between systems in $\overline{\Lambda\cC}(\fps{m_1},\dots,\fps{m_r})$ all of whose components are weak equivalences, then the map $\phi_{\cC}(f)$ is a natural transformation all of whose components are weak equivalences.

In this situation, we have an equivalence of $K$-theory spectra.

\begin{thrm}\label{thrm:comparisonisequivalence}
Suppose $\cC$ is a Waldhausen category with split cofibrations.  Then $\left|N\circ\diag\circ w\circ \phi_\cC\right|$ is an equivalence of $K$-theory spectra.
\end{thrm}
\begin{proof}
Since both $K$-theory constructions produce almost $\Omega$-spectra, the map in question, $\left|N\circ\diag\circ w\circ\phi_\cC\right|$, is a map of almost $\Omega$-spectra.  Hence it suffices to show that it is an equivalence at level 1, where we have the map
\[ \left|Nw(\overline{\cC}(S^1_\bullet))\right| \to \left|Nw(S^{(1)}_\bullet\cC)\right|\]
On the right-hand side, we have precisely Waldhausen's original $S_{\bullet}$ construction and on the left, the subcategory of weak equivalences in Segal's original construction.  Thus, this map is precisely the map shown to be an equivalence in \cite[\S 1.8]{Wald} in the case where the cofibrations in $\cC$ are split up to weak equivalence.
\end{proof}

Theorems \ref{thrm:nattransmultifunctors} and \ref{thrm:comparisonisequivalence} yield the following corollaries.  (This is the point of these results!)  For further details, see \cite[\S8--9]{EM2006}.  Recall our convention in \cref{itsalwayssymmetric} that ``multicategory'' means ``symmetric multicategory;'' correspondingly our operads are always symmetric as well.  In particular, any categorical $E_\infty$-operad $\cO$ is an example to which the following corollary applies.

\begin{cor}\label{cor:operadalg} Let $\cO$ be an operad in $\cat$ and let $\A$ be a $\cO$-algebra in Waldhausen categories.   Then there is a map of $\cO$ algebras in $\Spec(\scat)$  \[ \EMk(\Lambda \A)\to \waldk(\A).\]  After changing from the categorical enrichment to the simplicial enrichment, this produces a map of $N\cO$-algebra in spectra
  \[ N\EMk(\Lambda\A)\to N\waldk(\A)\]
  that is an equivalence when $\A$ has split cofibrations.
  \end{cor}
Here we use \cref{prop:changetosimpenrich} to identify  $N\EMk$ and $N\waldk$ as the usual versions of Elmendorf--Mandell and Waldhausen $K$-theory producing symmetric spectra in simplicial sets.

For example, this corollary implies that when $\A\in \mcwald$ is an $E_\infty$-algebra in $\mcwald$, the two $E_\infty$-ring spectra $\EMk(\Lambda \A)$ and $\waldk(\A)$ are equivalent \emph{as $E_\infty$-ring spectra} when $\A$ has split cofibrations.  

In fact, this corollary follows from a more general result. Recall that \cite{EM2006} defines small multicategories $\mathbf{M}$ that parametrize ring objects, $E_\infty$-objects and modules over ring objects.  They also show that for any such $\mathbf{M}$ there is a simplicial model structure on $\Spec^\mathbf{M}$ such that equivalences are objectwise stable equivalences of spectra \cite[Theorem 1.3]{EM2006}.  Theorems \ref{thrm:nattransmultifunctors} and \ref{thrm:comparisonisequivalence} then immediately imply the following corollary.

\begin{cor}\label{cor:multicategalgebra} Let $\mathbf{M}$ be a small categorically-enriched multicategory and let $\A\colon \mathbf{M}\to \mcwald$ be a multifunctor.   Then there is a commutative diagram 
\[\xymatrix{
\mathbf{M} \ar[r]^-{\A} & \mcwald \rtwocell<5>^{\EMk\circ \Lambda}_{\waldk} & \Spec(\scat_*).
}
\]
After applying the nerve at the level of morphisms to change to simplicial enrichment, we thus have a commutative diagram
\[\xymatrix{
N_\bullet\mathbf{M} \ar[r]^-{\A} & N_\bullet\mcwald \rtwocell<5>^{N\EMk\circ \Lambda}_{N\waldk} & \Spec.
}
\]
that is an equivalence in $\Spec^\mathbf{M}$ when the image of each object in $\A$ has split cofibrations.
\end{cor}

For example, if $\mathbf{M}$ is the multicategory of \cite[Def 2.4]{EM2006} that parametrizes modules over an $E_\infty$-object, then this multinatural transformation is an equivalence of module spectra.

Theorems \ref{thrm:nattransmultifunctors} and \ref{thrm:comparisonisequivalence} also yield equivalences of spectrally-enriched categories constructed from Waldhausen categories via these two approaches. If $\cC$ is a category enriched in $\mcwald$ and $F\colon \mcwald\to \Spec$ is any multifunctor, we denote by $F_\bullet\cC$ the spectrally-enriched category whose objects are the objects of $\cC$ and whose morphism spectra $(F_\bullet\cC)(c,d)$ are given by applying $F$ to the morphism Waldhausen categories in $\cC$: that is,
\[
(F_\bullet\cC)(c,d)=F(\cC(c,d)).
\]
Multifunctoriality of $F$ is precisely the condition needed to give a well-defined composition pairing of spectra in $F_\bullet\cC$.

\begin{cor} Let $\cC$ be a category enriched in $\mcwald$.   Then there is a spectrally-enriched functor 
\[(\EMk\circ \Lambda)_\bullet\cC\to \waldk_{\bullet}\cC\]
that is an equivalence of spectrally enriched categories if each morphism Waldhausen category in $\cC$ has split cofibrations.
\end{cor}

In the forthcoming \cite{waldmackey}, we use this corollary to show that Mackey functors of Waldhausen categories produce equivariant spectra, analogously to the construction of \cite{BO2015}.

\bibliographystyle{alpha}
\bibliography{references}

\end{document}